\documentclass{amsart}
\usepackage{graphicx} 
\usepackage{amsmath, amsfonts,amssymb,amsthm}
\usepackage[utf8]{inputenc}
\usepackage{mathtools}
\usepackage{biblatex}
\usepackage{tikz-cd}
\usepackage{hyperref}
\usepackage{comment}
\usepackage{enumerate}
\usepackage{enumitem}
\usepackage{multicol}
\usepackage{multirow}
\usepackage[normalem]{ulem}
\usepackage{microtype}
\usepackage{caption}
\usepackage{placeins}
\setlist[enumerate,1]{label=\arabic*.}
\title{Fixed Point Homogeneous Orbifolds with Positive Sectional Curvature}
\author{Jan Nienhaus and Dennis Wulle}

\let\emph\relax 
\DeclareTextFontCommand{\emph}{\bfseries\em}

\newcommand{\OO}{\mathcal{O}}
\newcommand{\Fr}{\mathrm{Fr}\,}
\renewcommand{\O}{\mathrm{O}}
\newcommand{\bb}{\mathbb}

\newcommand{\V}{\mathcal V}
\newcommand{\orb}{\mathrm{orb}}

\newcommand{\Id}{\rm Id}
\newcommand{\reg}{{\rm reg}}
\newcommand{\pr}{{\rm pr}}
\newcommand{\Fix}{{\rm Fix}}
\newcommand{\diag}{{\rm diag}}
\renewcommand{\S}{{\rm S}}
\newcommand{\T}{{\rm T}}

\newcommand{\Sph}{\mathbb{S}}

\newcommand{\gS}{\mathsf{S}}
\newcommand{\gC}{\mathsf{C}}
\newcommand{\gH}{\mathsf{H}}

\newcommand{\gG}{\mathsf{G}}
\newcommand{\rk}{{\rm rank}}

\newcommand{\Z}{\mathbb{Z}}
\newcommand{\R}{\mathbb{R}}
\newcommand{\C}{\mathbb{C}}
\newcommand{\HH}{\mathbb{H}}

\newcommand{\cp}{\mathbb{CP}}

\newcommand{\hp}{\mathbb{HP}}

\newcommand{\into}{\xhookrightarrow{}}
\newcommand{\acton}{\curvearrowright}
\newcommand{\Hom}{{\rm Hom}}
\newcommand{\eff}{{\rm eff}}

\newcommand{\SO}{{\rm SO}}
\newcommand{\Spin}{{\rm Spin}}
\newcommand{\codim}{{\rm codim}\,}

\newcommand{\Isom}{{\rm Isom}}
\newcommand{\std}{{\rm std}}

\newcommand{\U}{\mathsf{U}}

\newcommand{\Sp}{\mathsf{Sp}}

\newcommand{\SU}{\mathsf{SU}}

\newcommand{\NN}{\mathcal{N}}

\newcommand{\Ca}{{\rm Ca}}

\newtheorem{maintheorem}{Theorem}

\newtheorem{maincor}[maintheorem]{Corollary}

\newtheorem{theorem}{Theorem}[section]
\newtheorem{dfn}[theorem]{Definition}
\newtheorem*{theorem*}{Theorem}
\newtheorem{lem}[theorem]{Lemma}

\newtheorem{cor}[theorem]{Corollary}

\newtheorem{rem}[theorem]{Remark}
\newtheorem{exa}[theorem]{Example}

\newcommand{\bs}{\backslash}

\newcommand{\abs}[1]{\left\vert#1\right\vert}

\newcommand{\garrow}{\rightrightarrows}
\newcommand{\Groupoid}{G_1 \garrow G_0}
\newcommand{\setvert}{\,\vert\,}
\newcommand{\End}{{\rm End}}

\begin{document}



\begin{abstract}
    We classify fixed point homogeneous Riemannian orbifolds with positive sectional curvature up to equivariant diffeomorphism. As a corollary we obtain a classification of positively curved orbifolds with maximal symmetry rank.
\end{abstract}
\maketitle
\section{Introduction}

The construction and classification of Riemannian manifolds admitting positive curvature is one of the central problems in global differential geometry. Over decades the driving approach has been to impose some symmetry condition. Known as the \textit{Grove Symmetry Program}, this has led to new examples with positive \cite{Esch82,Grove2011, Dearricott2011} and nonnegative \cite{Goette2020, Wilking2002} sectional curvature, positive Einstein metrics \cite{Foscolo2017,Bohm1998, Nienhaus2025, Bohm2004}, and some intermediate positivity conditions \cite{Dominguez-Vazquez2023,DeVito2024}. On the structural side, some classification results have been achieved \cite{Grove1997-xz, gwz, Wulle2023} and topological results have been proven even under very mild symmetry assumptions \cite{Wilking2003,Wil06, Nienhaus22, Amann2010, DESSAI2007, Kennard13, Kennard2021-ge, Kennard2014}.

The main objects of this paper are Riemannian orbifolds. Orbifolds are singular spaces locally modeled on quotients of manifolds by finite groups. They naturally occur in moduli problems, e.g. in the famous work of Deligne and Mumford \cite{Deligne1969}, and in Gromov-Witten theory \cite{Dixon1985StringsOrbifolds, KlebanovWitten1998, DixonFriedanMartinecShenker1987, RuanCrepant, Beentjes2022}. Recently, topologists developed a structure theory of orbifolds \cite{Schwede2020Orbispaces, Pardon2023, Lange2022,dHF2metrics,dHFMetricsOnStacks, dHdMStackyGeodesics}, with notable results including Pardon's proof that every orbifold is a global quotient \cite{Pardon_2022} and BCR's proof of a variant of the crepant resolution conjecture \cite{Beentjes2022}. In Riemmanian geometry, orbifolds have been studied as singularity models for the Ricci flow \cite{Gianniotis2018, Deruelle2024, Li2025}, work with symmetry has dealt with quotients of spheres \cite{Gorodski2016} and biquotients \cite{WulleZarei26}. 
 
 Our first main result is
\begin{maintheorem}\label{main:A}
    Let $\OO$ be a Riemannian orbifold with positive sectional curvature  and $\pi_1^\orb(\OO) = 0$, carrying a fixed point homogeneous action by a compact connected Lie group $G$. Then $\OO$ is equivariantly diffeomorphic to $\bb S^n$, $\Ca \bb P^2$ or a weighted complex or quaternionic projective space with one of the actions in Table \ref{tab:examples}. 
\end{maintheorem}
An isometric action of a Lie group on a Riemannian orbifold is called fixed point homogeneous if the group acts transitively on the space of normal directions of some fixed point component. 

The orbifold fundamental group $\pi_1^\orb$ this is a finer invariant than $\pi_1$: $\pi_1^{orb}=0$ implies $\pi_1=0$ but the other implication fails in general. For details see Section \ref{sec:prelim}.

A weighted complex projective space is an orbifold $\cp[p_0, \ldots, p_n]$ defined as the quotient $\bb S^{2n+1}/\S^1$ with $z\in \S^1$ acting via $\diag(z^{p_0}, \ldots, z^{p_n})$, where $p_0, \ldots, p_n$ are nonzero integers.
All of these spaces are simply connected and fixed point homogeneous, and there are infinitely many homotopy types among them in each even dimension $\ge 4$ \cite{Kawasaki1973}. 

One constructs quaternionic weighted projective spaces analogously using quaternionic $\S^3$-representations (see \ref{subsection:wqtproj}). There are only finitely many of these in each dimension, and not all of them occur in Theorem \ref{main:A}, see Section \ref{subsection:quaternionic}.

All weighted projective spaces carry a canonical positively curved metric $g_{std}$ induced by the submersion from the sphere. For each space $\OO$ of Theorem \ref{main:A}, we will call the isometry group of its canonical metric $I_{std}(\OO)$. 

We note that Theorem \ref{main:A} recovers not only the orbifold, but also puts the action of $G$ into an explicit standard form. For a full account of the examples, see Section \ref{sec:examples}.

The classification of fixed point homogeneous actions on manifolds with positive sectional curvature was  achieved by Grove and Searle in \cite{Grove1997-xz}, where only $\bb S^n$, $\Ca \bb P^2$,  $\cp^n$ and $\hp^n$ occur, in short the CROSSes, if $\pi_1 =0$. 

\begin{rem} We  do not make any implicit assumption of effectivity in Theorem \ref{main:A}, as is done in many differential geometry papers on orbifolds. An orbifold is ineffective if the isotropy group is nontrivial at all points. 
\end{rem}
Dealing with ineffective orbifolds comes with difficulties, but cannot be avoided if one is interested in applications: The most common use case of the manifold version of Theorem \ref{main:A} is in further classification results (see above), where one applies it to submanifolds of some given ambient space. However, even for effective orbifolds, suborbifolds are not guaranteed to inherit this property. One may try to resolve this by taking the canonical effective quotient of the suborbifold, but this process loses important embedding information carried by the ineffective structure.

For orbifolds with fundamental group we obtain

\begin{maintheorem}\label{main:fundamentalgroup}
    Let $\overline\OO$ be a Riemannian orbifold with positive sectional curvature, carrying a fixed point homogeneous action by a compact connected Lie group $G$, and let $\OO$ be its universal cover. Then $\pi_1^{orb}(\overline\OO)$ is finite and up to equivariant diffeomorphism the action of $\pi_1^{orb}(\overline \OO)$ on $\OO$ factors through the centralizer of $G$ in $I_{std}(\OO)$.
\end{maintheorem}
Explicitly, this means there is a homomorphism $\pi_1^{orb}(\overline\OO)\to I_{std}(\OO)$ whose image commutes with $G\subset I_{std}(\OO)$ such that $\overline\OO$ is equivariantly diffeomorphic to $\OO/\pi_1^{orb}(\overline \OO)$, where the action of $\pi_1^{orb}(\overline\OO)$ on $\OO$ is the one induced by this homomorphism.
 In general, dividing $\OO$ by $\pi_1^\orb(\overline\OO)$ has to be understood in the sense of groupoids, see Example \ref{exa:actiongroupoid}. If $\overline{\OO}$ happens to be effective, then the homomorphism is injective and one recovers a usual quotient by a subgroup of $I_{std}(\OO)$.

In particular, Theorem \ref{main:fundamentalgroup} rules out any exotic group actions on the models from being compatible with any fixed point homogeneous structure. Moreover, note that the orbifold case allows a lot more fundamental groups than the manifold case since the fundamental group is not required to act freely on the universal cover. For example, since every finite group has finite-dimensional representations, among effective orbifolds with universal covers $\cp^n$ every finite group occurs as $\pi_1^{orb}$, while for positively curved \textit{manifolds} of dimension $2n$ the fundamental group can only be trivial or $\Z_2$.

Furthermore, as corollary to Theorem \ref{main:A} we get a classification of positively curved orbifolds whose isometry group has maximal possible rank:
\begin{maincor}\label{main:maxsymmetryrank}
    Let $\OO^n$ be a positively curved simply connected Riemannian orbifold with maximal symmetry rank, i.e. such that the rank of the isometry group is exactly $\left\lfloor \frac{n+1}{2} \right\rfloor$. Then $\OO$ is diffeomorphic to $\bb S^n$ or a possibly ineffective complex weighted projective space.
\end{maincor}
A classification up to homeomorphism of the underlying space of $\OO$ was previously known as a corollary of the Alexandrov space result by Harvey and Searle \cite{HarSea17}. Here we obtain a classification up to diffeomorphism of the orbifold. Even disregarding questions of effectivity, recovery of the orbifold structure is a nontrivial problem as non-equivalent orbifolds may have homeomorphic underlying spaces. We actually prove a slightly sharper result (Corollary \ref{cor:fphom:maxsymmetry}): The diffeomorphism can be made equivariant with respect to the action of a maximal torus of the isometry group of $\OO$.
Here too the fundamental groups may be described as in Theorem \ref{main:fundamentalgroup}. However, the resulting list is much less rich since commuting with a maximal torus of $I_{std}(\OO)$ is a very restrictive condition.\\

The idea of the proof of Theorem \ref{main:A} is to first consider the effective case. Here the strategy is the same as in the manifold case \cite{Grove1997-xz}: Fixing a fixed point component $F$ where $G$ acts transitively on the normal space to $F$, first one uses the Alexandrov geometry of $\OO/G$ to show that there is a unique orbit $Gp$ at maximal distance to $F$. This gives $\OO$ the structure of a double disk bundle, i.e. $\OO$ is the union of the normal bundles of $F$ and of $Gp$ along their common boundary. In the manifold case one then proves a recovery result showing that $\OO=M$ is already uniquely determined by the bundle at $Gp$, and proceeds to classify from there. However, in the orbifold case such a recovery result simply isn't true, as exhibited by infinite families of weighted complex projective spaces where the bundles at $Gp$ are all equivariantly diffeomorphic. To find additional invariants that distinguish these classes, one has to consider the structure near $F$ instead of working only at $Gp$. The second main difficulty is recovery of the structure of the normal bundle at $Gp$. This involves extension problems involving the local groups of the orbifolds and the isotropy groups of the group action. Finally, after recovering the effective theorem, to conclude the general case we proceed to classify the ineffective extensions of the examples and show they are all compatible with fixed point homogeneous actions. This requires us to switch to the groupoid point of view: This is very much topological in nature; we show that in our setting ineffective extensions are classified by $H^2_{orb}(\OO)$ in suitable coefficients, which in the end reduces us to computations of orbifold homotopy groups using the fibrations associated to the examples. One then concludes Theorem \ref{main:A} by showing that all classes in the cohomological classification are covered by the examples. 

For Theorem \ref{main:fundamentalgroup}, one notes that the property of being fixed point homogeneous passes to the universal cover, reducing the problem to understanding actions of a finite group $L$ compatible with a fixed point homogeneous structure. There is a dichotomy here: If $L$ fixes the orbit $Gp$ as a set, one obtains a linear action of $L$ on the normal bundle of $Gp$. Then one has to recover how $L$ acts on the points of $Gp$. On the other hand, if $L$ does not preserve $Gp$, then one can show $F$ and $Gp$ must both be points, so $\OO$ is cohomogeneity $1$. One can then run an argument similar to the other case on the orbit in the middle, which must be fixed by $L$.

\subsection*{Structure}

After gathering some preliminaries in Section \ref{sec:prelim}, especially regarding some of the subtleties around orbifolds, the main body of the paper starts with a discussion of examples, in particular those that occur in Theorem \ref{main:A} but not in the manifold version. This will take up Section \ref{sec:examples}. The core work of the paper is done in Section \ref{sec:effective}, where the main structure results are proven in the effective case. Section \ref{sec:ineffective} deals then with the ineffective case. Here the description of orbifolds in terms of charts breaks down and one has to switch to the language of groupoids, see Section \ref{sec:groupoids}.

\subsection*{Acknowledgments}
The first named author would like to thank the Cluster of Mathematics M\"unster for enabling research stays in M\"unster in the summers of 2025 and 2026, where part of this work was carried out.

The second named author acknowledges support by the Alexander von Humboldt Foundation through Gustav Holzegel’s Alexander von Humboldt Professorship endowed by the Federal Ministry of Education and Research and by the Deutsche Forschungsgemeinschaft (DFG, German Research Foundation) under Germany's Excellence Strategy EXC 2044/2 –390685587, Mathematics Münster: Dynamics–Geometry–Structure.

The authors would also like to thank Christoph B\"ohm for many helpful comments regarding the exposition of the paper. 

\subsection*{AI disclosure}
The authors had access to publicly available AI systems. The paper contains neither mathematical results nor text generated by AI, but AI has been useful in identifying relevant work in the existing literature.
\tableofcontents
\section{Preliminaries}\label{sec:prelim}

Anyone unfamiliar who has tried to get into orbifolds quickly realizes a simple fact: There is a zoo of definitions, most of which are presumably equivalent in relevant cases, but which are so different from one another that large sections of the literature are almost inaccessible to people who don't have the right background. In practice, the approaches to orbifolds split into three camps:
\begin{enumerate}
    \item Orbifolds via charts \cite{Sat56, Car22} 

    This is the original, and "na\"ive" way of defining orbifolds, using charts modeled on quotients of $\R^n$ by finite group actions, analogously as one would manifolds. This is generally preferred by people coming to orbifolds from differential geometry. However, there are some subtleties that arise, in particular regarding smoothness of maps and the fact that the theory doesn't handle the ineffective case, which have led to

    \item Orbifolds as groupoids \cite{EhrLieCats,moerdijk2002orbifoldsgroupoidsintroduction, ALR07, dHF2metrics}

    For the price of a more technical approach, this language is able to resolve most of the problems of the classical definition. One of the significant features of this approach is that one orbifold may have different groupoid models that look very different from one another, with the relevant notion being that of "Morita equivalence". This can be useful in practical situations, as the model corresponding to finite groups acting on charts may not be well-adapted for, for example, best describing the quotient of a manifold by an action of a continuous group, but really one would prefer an orbifold to be a single object:

    \item Orbifolds as differentiable stacks \cite{Behrend2006DifferentiableSA, HeinlothNotesOnStacks, LermanOrbisAsStacks, dHFMetricsOnStacks}

    If equivalent atlases of charts are different concrete models of the same abstract manifold, groupoids are such "concrete models" of stacks. This is probably the most philosophically correct approach, but the language needed to formulate the definitions is highly categorical in nature and the theory is in a very early stage of development when it comes to questions of geometry. For example, a notion of geodesic has only been introduced in 2022 \cite{dHdMStackyGeodesics}.
\end{enumerate}

Our strategy of proof is as follows: We will deal separately with the effective and ineffective cases. For the effective version of the theorem, we will work in the language of charts, where it is easy to make sense of everything from a geometric point of view. Afterwards, we will switch to the groupoid perspective to find the possible ineffective extensions. 

\subsection{Effective Orbifolds}\label{sec:prelim:effective}
For our exposition, we follow Chapter 2 of \cite{KL14}. 
\begin{dfn}[Local models]
\mbox{}
\begin{enumerate}
    \item   A \emph{local model} is a pair $(\hat U, \Gamma)$, where $\hat U\subset \bb R^n$ is a connected open subset and $\Gamma$ a finite group acting effectively on $\hat U$ from the right. We denote  $U = \hat U/\Gamma$.
    \item A \emph{smooth map between local models} $f = (\hat f,\rho) \colon (\hat U, \Gamma)\to (\hat V, \Gamma^\prime)$ is given by a smooth map $\hat f \colon \hat U \to \hat V$ and a homomorphism $\rho \colon \Gamma \to \Gamma^\prime$, such that $\hat f$ is $\rho $-equivariant, i.e. $\hat f(x \cdot \gamma) = f(x) \cdot \rho(\gamma)$. 
    \item An \emph{embedding} of local models is a smooth map $f = (\hat f,\rho) \colon (\hat U, \Gamma)\to (\hat V, \Gamma^\prime)$ between local models, such that $\hat f $ is an embedding. 
\end{enumerate}
\end{dfn}

\begin{dfn}[Orbifold Atlas]\label{def:orbiatlas} Let $X$ be a paracompact topological Hausdorff space. An \emph{orbifold atlas} $\mathcal A$ for $X$ consists of 
\begin{enumerate}
    \item An open covering $\mathcal{U} = \lbrace U_i\setvert i \in I\rbrace $ of $X$,
    \item local models $\hat {\mathcal U} = \lbrace (\hat U_i, \Gamma_i)\setvert i \in I \rbrace $ with $\hat U_i \subset \bb R^n$,
    \item homeomorphisms $\phi_i \colon U_i \to \hat U_i/\Gamma_i$, such that
    \item let $U_1, U_2 \in \mathcal U$ and $p \in U_1 \cap U_2$, then there is a local model $(\hat U_3, \Gamma_3) \in \hat{\mathcal U}$ with $p \in U_3  = \hat U_3/\Gamma_3$ and embeddings $(\hat U_3, \Gamma_3) \to (\hat U_1,\Gamma_1)$ and $(\hat U_3, \Gamma_3) \to (\hat U_2,\Gamma_2)$.
 \end{enumerate}
 Two Atlases for $X$ are \emph{equivalent} if they are both contained in a third Atlas.
\end{dfn}

\begin{dfn}[Orbifold]
An \emph{effective orbifold} $\OO$ is a pair $(X,[\mathcal A])$, where $X$ is a paracompact topological Hausdorff space and $[\mathcal A]$ and equivalence class of atlases. We write $\vert \OO\vert \coloneqq X$  
\end{dfn}

Similarly it is possible to define effective orbifolds with boundary by assuming, that the orbifold is  locally modelled on open subsets $\hat U \subset [0, \infty) \times \bb R^{n-1}$. The boundary $\partial \OO$ is then an $n-1$-dimensional orbifold, where $\vert \partial\OO \vert$ consists of points, whose local lifts lie in $\lbrace 0 \rbrace \times \bb R^{n-1}$. Note that it is possible that $\partial \OO = \emptyset$, while $\partial \vert \OO \vert \neq \emptyset$. 

\begin{exa}
    Consider $X=[0, \infty)$. We may equip $X$ with two orbifold structures: $\OO_1:[0,\infty)=[0,\infty)$ as a nonsingular orbifold with boundary, and $\OO_2:[0,\infty)=\R/\Z_2$ as an orbifold without boundary that has isotropy $\Z_2$ at the origin.
\end{exa}

This example shows that the structure of an orbifold is strictly richer than that of its underlying topological space. In higher dimensions, one can produce such examples even when disallowing boundary.

\begin{dfn}\mbox{}
\begin{enumerate}
    \item Let $p \in \abs{\OO}$ and $(\hat U, \Gamma)$ a local model around $p$. Let $\hat p \in \hat U$ be a lift of $p$, then the \emph{local group} $\Gamma_p$ at $p$ is the stabilizer $\Gamma_p = \lbrace \gamma \in \Gamma \setvert p \cdot \gamma = p \rbrace$. Its isomorphism class is independent of the choices made.     
    \item We denote by $\OO_\reg$ all points $p\in \abs{\OO}$ such that $\Gamma_p = \lbrace e \rbrace$. $\OO_\reg$ inherits the structure of a smooth manifold and forms an open and dense subset in $\OO$.
    \item The \emph{singular set} $\Sigma(\OO)$ is given by all points $p \in \vert \OO \vert \bs \OO_\reg$
\end{enumerate}
\end{dfn}

\begin{exa}\label{intro:quotients}
    Let $M$ be a manifold and $G$ a compact Lie group, such that $G$ acts almost freely on $M$. If the principal isotropy group of this action is trivial, then the quotient $\OO = M/G$ has the structure of an effective orbifold. The regular orbits form the regular set of $\OO$ and the local groups correspond to the isotropy groups of the $G$-action. If the principal isotropy group is not trivial, then the quotient is a so-called \emph{ineffective orbifold}. We will discuss the notion of ineffective orbifolds in the next section.
\end{exa}

\begin{dfn}[smooth maps\footnote{We choose the na\"ive / bad definition of smooth for its simplicity. Smooth maps in this sense are not automatically smooth/good in the general (groupoid) sense. To illustrate some of the subtleties, orbifold vector bundles do not always have pullbacks under this definition. For a study of the different notions see \cite{BBSmoothnessNotions}. Note that for diffeomorphisms in particular the definitions agree \cite{KL14}.}]
 A \emph{smooth map} $f \colon \OO \to \mathcal P$ between orbifolds is given by a continouus map $\abs{f} \colon \abs{\OO} \to \abs{\mathcal P}$ with the property that for each $p \in \abs \OO$ there are 
 \begin{enumerate}
     \item Local models $(\hat U, \Gamma)$ and $(\hat V, \Gamma^\prime)$ around $p$ and $f(p)$ and a smooth map $\hat f \colon (\hat U, \Gamma) \to (\hat V, \Gamma^\prime)$ of local models, such that the following diagram commutes
     \begin{center}
              \begin{tikzcd}
         \hat U \arrow[r, "\hat f"] \arrow[d]& \hat V \arrow[d] \\
         U \arrow[r, "\abs{f}"] & V
     \end{tikzcd}
     \end{center}
 \end{enumerate}
 A \emph{diffeomorphism} is a smooth map   $f \colon \OO \to \mathcal P$ between orbifolds  with smooth inverse. 
\end{dfn}

\begin{dfn}[Group action]
 Let $G$ be a Lie group and $\OO$ be an orbifold. A smooth left action of $G$ on $\OO$ is a smooth map of orbifolds $$\Phi \colon G \times \OO \to \OO$$    
 such that the induced map $\abs{\Phi} \colon G \times \abs\OO \to \abs{\OO}$ is a continuous left action of topological spaces. 
\end{dfn}
Note that with this notion an element $g \in G$ acts as a diffeomorphism. 

\begin{dfn}
    An \emph{orbivector bundle} $\mathcal V \to \OO$ is locally modelled on $(V\times \hat U,\Gamma)$, where $(\hat U, \Gamma)$ is a local model of $\OO$.
\end{dfn}

 \begin{dfn}
     A smooth map $f \colon \OO \to \mathcal P$ gives rise to a \emph{differential} $Tf \colon T\OO \to T \mathcal P$ between the orbi tangent bundles in the following way: Let $p \in \abs\OO$, $(\hat U, \Gamma)$ a local model around $p$ and $(\hat V, \Gamma^\prime)$ a local model around $f(p)$. This gives rise to a $\Gamma_p$-equivariant linear map $T\hat f \colon T\hat U \to T\hat V$, which defines the differential. \\
     We say that $f$ is an \emph{immersion} (resp. \emph{submersion}) if $T_pf \colon T_p \OO \to T_{f(p)} \mathcal P$ is injective (resp. surjective) for all $p \in \abs\OO$. 
 \end{dfn}

\begin{dfn}
    A \emph{suborbifold} of $\OO$ is an orbifold $\OO^\prime$ together with an immersion $f\colon \OO^\prime \to \OO$ for which $\abs{f}$ maps $\abs{\OO^\prime}$ homeomorphically to its image. 
\end{dfn}

 As for manifolds we may define Riemannian metrics on an orbifold $\OO$ on the charts, asking in addition that the chart metrics are invariant under the action on the local model. Since orbifolds admit partitions of unity \cite{KL14}, every orbifold admits a Riemannian metric.     We note, that the normal bundle of a suborbifold is again an orbivectorbundle.

 \begin{theorem}\cite{Sat56}
     Let $(\OO,g)$ be an effective Riemannian orbifold. Then the \emph{Frame bundle} $\Fr\OO$ is a manifold, $\O(n)$ acts almost freely on $\Fr\OO$, and $\OO \cong \Fr\OO/\O(n)$.
 \end{theorem}
\begin{rem}
    Let $\OO$ be a Riemannian orbifold and $\mathcal N$ be a suborbifold. One would expect that $\Fr\OO\vert _{\mathcal N}/\O(n)$ always recovers the orbifold structure of $\mathcal N$, but this is not the case, since the $\O(n)$-action on $\Fr\OO\vert_{\mathcal{N}}$ might have nontrivial principal isotropy group $\Gamma$. Only after dividing out $\Gamma$ at the local groups we arrive at the effective orbifold $\mathcal N$. This illustrates again the need for a notion of orbifolds which have isotropy at every point. 
\end{rem}
There is a slice theorem for orbifolds, since one can lift the action to the frame bundle, where the slice theorem for manifolds applies. 

The frame bundle allows us to define topological invariants for orbifolds that do not only rely on the underlying space. Let $\OO$ be a Riemannian orbifold and $$\O (n) \to E\O(n) \to B\O(n)$$ the classifying bundle of $O(n)$. We consider the Borel construction
$$B\OO \coloneqq E\O(n) \times_{\O(n)} \Fr\OO.$$
Now we may define the orbifold invariants to be the topological invariants of the space $B\OO$. In particular, we have 
\begin{enumerate}
    \item Orbifold homotopy groups: $\pi_n^\orb(\OO) \coloneqq \pi_n(B\OO)$,
    \item Orbifold homology: $H_\ast^\orb(\OO,R) \coloneqq H_\ast(B\OO,R)$,
    \item Orbifold cohomology: $H_\orb ^\ast(\OO,R)\coloneqq H^\ast(B\OO,R)$, \ldots 
\end{enumerate}
We note that there is a natural map $\pi_1^\orb(\OO) \to \pi_1(\abs\OO)$, which is always surjective. 

\subsection{Orbifolds as Groupoids}\label{sec:groupoids}
This subsection is the foundation for the ineffective case covered in section \ref{sec:ineffective}. Readers, that are only interested in the effective case and the geometric can directly skip to the following section.
Before we set out to outline the general theory, we have a look at an example as to why a chart-based model struggles:

Take the simplest case, where we have a group $G$ acting on a closed manifold $M$ with every isotropy finite. Then, in the general as in the effective case, we want to define the topological invariants of the "quotient orbifold" so that they equal those of the Borel construction $M\times_G EG$.\footnote{In the chart model, the frame bundle is exactly the construction that brings every $\OO$ into this form} Now look at the following two examples:

$$\S^1\times \Z_k\acton \bb S^3, (z,a)\star p=zp,\text{ and } \S^1\acton \bb S^3, z\star p=z^kp.$$

Both of these have isotropy $\Z_k$ at every point, and (as manifolds) the quotient of each is a smooth $S^2$. However, $\pi_1(S^3\times_G EG)=\Z_k$ for the first example, and, perhaps surprisingly, $\pi_1(S^3\times_G EG)=0$ in the second. 

This means any good general definition of orbifold must be able to distinguish the two examples. On the other hand, the underlying space is simply a smooth $S^2$, and it is hard to come up with charts where it is possible to act trivially by $\Z_k$ on each chart, but in two different ways. The solution is to move away from basing our construction around an "underlying topological space", and instead carry along a bit more data:

\begin{dfn}\mbox{}
\begin{enumerate}
    \item   A \emph{topological groupoid} $G_1 \garrow G_0$ is a small category with a space of objects $G_0$ and a space of arrows $G_1$, so that each arrow is invertible and such that the five following structure maps are continuous:
    \begin{enumerate}
        \item The source map $s \colon G_1 \to G_0$, assigning each arrow $g \in G_1$ its source $s(g) \in G_0$
        \item The target map $t \colon G_1 \to G_0$, assigning each arrow $g \in G_1$ its target $t(g) \in G_0$. 
        \item The composition map $m \colon G_1 \,{}_s\times_tG_1 \to G_0$, which is associative.
        \item The unit map $\Id \colon G_0\to G_1$ assigning to each $g \in G_0$ the identity element $\Id_g \in G_1$.
        \item The inverse map $i \colon G_1 \to G_1$, assigning to each arrow $g \in G_1$ its inverse $i(g)=g^{-1}$
    \end{enumerate}
    \item A \emph{smooth Groupoid} or \emph{Lie Groupoid} is a topological groupoid $G_1\garrow G_0$, such that $G_0$ and $G_0$ are smooth manifolds, the structure maps are smooth and the source and target maps are submersions.
\end{enumerate}
\end{dfn}

While this is intimidating at first glance, the intuition is simple: $G_0$ gives us points, and an arrow (that is, an element of $G_1$) with source $p$ and target $q$ indicates that $p$ and $q$ should be identified. A few examples:

\begin{exa}
    Let $M$ be a smooth manifold. Then $M$ can be viewed as a Lie groupoid via $M\garrow M$ with source and target map equal to the identity
\end{exa}

Slightly less trivially, the atlas definition of a manifold also fits into this framework

\begin{exa}\label{exa:manifoldatlas}
    Let $\{\varphi_\alpha :U_\alpha\subset M \to V_\alpha\subset \R^n\}$ be an atlas of a manifold $M$ with transition maps $\varphi_{\alpha \beta}:V_{\alpha \beta}\to V_{\beta \alpha}$, where $V_{\alpha \beta}:=\varphi_\alpha(U_\alpha \cap U_\beta)$.

    Then $\bigcup V_{\alpha \beta}\garrow \bigcup V_\alpha$, with $s:V_{\alpha \beta}\to V_\alpha$ the inclusion and $t:V_{\alpha \beta}\to V_\beta$ equal to $\varphi_{\alpha \beta}$, is a Lie groupoid equivalent (see def \ref{def:moritaeq}) to $M\garrow M$.
\end{exa}

\begin{exa}\label{exa:actiongroupoid}
    Let $M$ be a manifold and $G$ a compact Lie group acting smoothly on $M$. Then $M/G \coloneqq G\times M \garrow M$, with source map $s(g,p) = p$ and target map $t(g,p) =gp$, is a Lie groupoid  called the \emph{action groupoid}.
\end{exa}

For the purpose of encoding an orbifold in the sense of section \ref{sec:prelim:effective} as a groupoid, there are two prime ways: Using the action groupoid associated to the action $\O(n)\acton \Fr(\OO)$ on the frame bundle, or straight from an atlas, in the spirit of example \ref{exa:manifoldatlas}

\begin{exa}
    Starting with an orbifold atlas $\{(\hat U_i, \Gamma_i)\}$, set $G_0=\bigcup\hat U_i$. The starting point for $G_1$ is the union of the action groupoids $\bigcup \Gamma_i\times\hat U_i\garrow \bigcup\hat U_i$. As in the manifold case, one now needs to add arrows between points in different charts corresponding to the same geometric point. These are built from the embeddings of the last condition in definition \ref{def:orbiatlas}, but the details are a bit technical. See \cite{MoerdijkPronk97}.
\end{exa}

Each variant has its upsides. While often it is much easier to compute things for action groupoids, the groupoid obtained from an atlas has other nice properties, for example that there are only finitely many arrows between any given two points. More precisely, it is proper and \'etale (see definition \ref{def:properetale}).

Finally, we need some good notion of equivalence. For example, for $G\acton G\times M$ acting by multiplication on the first factor, the quotient should not depend on $G$, and more generally, for any orbifold the different ways of constructing a groupoid representing it should all be equivalent.

\begin{dfn}\label{def:moritaeq}
\mbox{}
\begin{enumerate}
    \item A morphism\footnote{the obvious definition is the correct one} of groupoids $\Phi \colon \Groupoid \to H_1 \garrow H_0$ is called an equivalence if
    \begin{enumerate}
        \item The map $$t \circ \pr_1 \colon H_1 {}_s \times_\Phi G_0 \to H_0$$ defined on $H_1 {}_s \times_\Phi G_0 = \lbrace (h,y)\vert h \in H_1, y \in G_0, s(h) = \Phi(y)\rbrace $ is a surjective submersion.
        \item The square
        \begin{center}
            \begin{tikzcd}
              G_1 \arrow[d,"{(s,t)}" left] \arrow[r,"\Phi"] &  H_1 \arrow[d, "{(s,t)}"] \\
                    G_0 \times G_0 \arrow[r, "\Phi \times \Phi"]  & H_0 \times H_0
            \end{tikzcd}
        \end{center} 
        is a fibered product of manifolds.\footnote{Explicitly, the induced map $G_1 \to \{(h,x,y)\in H_1\times G_0\times G_0 \setvert h:\Phi x\to\Phi y\}$ is a diffeomorphism}
    \end{enumerate}
    \item Two groupoids $G, H$ are \emph{Morita-equivalent} if there is a diagram of groupoids $G \leftarrow Z \rightarrow H$ where both arrows are equivalences.
\end{enumerate}
\end{dfn}

To finally define orbifolds in this language, we need two adjectives:

\begin{dfn}\label{def:properetale}
    Let $\Groupoid$ be a Lie groupoid. 
    \begin{enumerate}
        \item $\Groupoid$ is \emph{proper} if $(s,t)\colon G_1 \to G_0 \times G_0$ is a proper map.
        \item $\Groupoid$ is \emph{\'etale} if $s,t$ are local diffeomorphisms
    \end{enumerate}
\end{dfn}

\begin{dfn}
   An \emph{orbifold} $\OO$ is a Lie groupoid Morita equivalent to\footnote{definitions differ as to whether an orbifold should be a proper \'etale, be equivalent to a proper \'etale, or be the equivalence class of a proper \'etale Lie groupoid. We will often describe orbifolds as action groupoids of continuous groups.} a proper \'etale Lie groupoid $G_1\garrow G_0$. 
\end{dfn} 

Note that being proper is invariant under Morita equivalence, but being \'etale is not. For example action groupoids are only \'etale if the group is discrete.

\begin{exa}
    Let $G \acton M$ be an almost free group action by a compact Lie group. Then the action groupoid defines an orbifold.
\end{exa}
\begin{proof}
    By repeated use of the slice theorem, we may cover $M$ by a locally finite union of open sets, each of which is a neighborhood of some orbit $G p_i$, having the form $U_i \cong (G\times D_i)/G_{p_i}$, where $D_i$ is a disk normal to $G p_i$ at $p_i$.
    Let now $\hat \OO$ be the subgroupoid of the action groupoid with $\hat G_0=\bigcup D_i$, with all arrows with admissible endpoints (explicitly, $\hat G_1=\{(p,g)|p\in \hat{G_0}, gp\in \hat{G_o}\}$). Then $\hat \OO$ is \'etale and the inclusion $\hat\OO \into \OO$ is an equivalence.
\end{proof}

\begin{theorem}[\cite{Pardon_2022}] Any orbifold groupoid is morita equivalent to an action groupoid by an almost free Lie group action.    
\end{theorem}

At this point we should say a few words about recovery of the classical data. 

\begin{dfn}
    The \emph{geometric realization}, or coarse quotient, of a Lie groupoid $\OO$ is the topological space $\abs{\OO}=G_0/\sim$, where $x\sim y:\iff$ there is an arrow $G_1 \ni g:x\to y$.

    For $\bar x \in \vert \OO\vert$ let $x \in G_0$ be a representative. Then the \emph{local group} is $\Gamma_x \coloneqq G_x$. The set of principal $G_1$-orbits in $G_0$ is called the \emph{regular part} of $\OO$. All other points are called \emph{singular}. Let $p \in \OO_\reg$. We call $\OO$ \emph{effective} if $\Gamma_p = 1$ and otherwise \emph{ineffective}.
\end{dfn}

\begin{dfn}
    A strict action of a Lie group $L$ on a Lie groupoid $\Groupoid$ consists of two action $L \times G_i \to G_i$ for $i = 0,1$, such that the source and target maps are equivariant.  
\end{dfn}
A strict action on a groupoid induces a topological action on the coarse quotient. Riemmanian metrics on groupoids have been defined in \cite{dHF2metrics}. These induce a metric on the coarse quotient. For this work it is only relevant, that the full quotient space $\vert \OO\vert/G$ has positive sectional curvature in the Alexandrov sense. Hence, we will not discuss further curvature conditions. Strict isometric actions on Riemannian groupoids have been discussed in \cite{Herrera-Carmona2023}. 

Since passing to the frame bundle of the geometric realization runs into the issue of forgetting all ineffective data, we also need a different way to obtain our topological invariants. However, everything turns out to be fine

\begin{theorem*}
    For every groupoid $\OO$, there is a classifying space\footnote{The nomenclature comes from the fact that for groupoids $(G\garrow\{\star\})$, one recovers the classical classifying space of the group $G$} $B\OO$. If two groupoids are Morita-equivalent, then their classifying spaces are homotopy equivalent. In particular, if $\OO$ is equivalent to an action groupoid $G\acton M$, $B\OO$ is homotopy equivalent to the Borel construction of any such action.
\end{theorem*}

The actual construction is very general, and works for categories much more general than Lie groupoids. However, this is more than enough for our purposes, which are to define invariants:

\begin{dfn}
    Given an orbifold $\OO$, the orbifold invariants $\pi_*^{orb}(\OO)$, $H^*_{orb}(\OO)$, etc. are defined to be the regular invariants $\pi_*(B\OO), H^*_{\orb}(B\OO)$, etc. of its classifying space.
\end{dfn}

\section{Examples}\label{sec:examples}
In this section we discuss the spaces occuring in Theorem \ref{main:A}, with an emphasis placed on the weighted complex and quaternionic projective spaces, which are the only new examples in Theorem \ref{main:A} in comparison to the manifold classification \cite{Grove1997-xz}. 

\subsection{Spheres and the Cayley Plane}
These examples are well-known, but we include a description of the possible fixed point homogeneous actions for the sake of having everything in one place.

\begin{exa}\label{exa:fphspheres}
    Let $G\acton \bb S^k$ be any transitive action, and $l\ge 1$. Then one obtains a fixed point homogeneous action $G\acton \bb S^{k+l}$ by letting $G$ act on the first $k+1$ coordinates
\end{exa}

By the classification in the manifold case, every fixed point homogeneous action on $\bb S^n$ can be brought into this form. The transitive actions on spheres are classified, the list is given in Table \ref{tab:fixhom:homospheres}.
{
\renewcommand{\arraystretch}{1.3}
\begin{table}[!h]
      \begin{center}
          \begin{tabular}{|c|c|c|c|c|c|}
\hline  &$\gG$          &$\gH$ &$\gG/\gH$&Kernel $\gC$ &$N(\gH)/\gH$           \\
\hline \hline
1. &$\SO(n+1)$ &$\SO(n)$ &$\Sph^{n}$ &$\{e\}$& $\Z_2$ (for $n\ge 2$)\\
\hline

2.&$\SU(n+1)$ & $\SU(n)$& $\Sph^{2n+1}$& $\{e\}$ &$\gS^1$ (for $n\ge 2$)  \\
\hline

2.1.&$\U(n+1)$ & $\U(n)$ & $\Sph^{2n+1}$ &$\{e\}$ & $\gS^1$ \\
\hline

3.&$\Sp(n+1)$ & $\Sp(n)$ &  $\Sph^{4n+3}$& $\{e\}$ & $\gS^3$ \\
\hline

3.1.&$\Sp(n+1)\U(1)$ & $\Sp(n)\Delta\U(1)$&$\Sph^{4n+3}$& $\Delta \Z_2$&$\gS^1$
 \\
\hline
3.2.&$\Sp(n+1)\Sp(1)$ & $\Sp(n)\Delta\Sp(1)$& $\Sph^{4n+3}$&$\Delta \Z_2$& $\Z_2$
\\
\hline
4.&$\Spin(9)$   &$\Spin(7)$& $\Sph^{15}$& $\{e\}$ & $\Z_2$  \\
\hline

5.&$\Spin(7)$   &$\gG_2$ &$\Sph^7$  & $\{e\}$ & $\Z_2$ \\
\hline

6.& $\gG_2$   &$\SU(3)$  &$\Sph^6$  &$\{e\}$  & $\Z_2$\\
\hline
          \end{tabular}
       \end{center}
      \caption{Transitive actions on spheres.}
      \label{tab:fixhom:homospheres}
\end{table}
}

For the Cayley plane, one only has one fixed point homogeneous action:

\begin{exa}\label{exa:fphcayley}
    Consider the Cayley plane $\Ca\bb P^2 = F_4/\Spin(9)$. The natural action of $\Spin(9)$ on the left clearly fixes a point. This is the unique fixed point homogeneous action on $\Ca\bb P^2$.
\end{exa}

In fact, this action and its iterated isotropy representations are the source of the exceptional homogeneous spheres 4. 5. and 6. in Table \ref{tab:fixhom:homospheres}.

\subsection{Weighted projective spaces}
\subsubsection{Weighted complex projective spaces}
Let $p = (p_0,\ldots, p_n) \in \bb Z^{n+1}$ with all $p_i$ nonzero. The homomorphism \begin{align*}
    \rho_p \colon \S^1 \to \U(n+1), z \mapsto \diag(z^{p_0}, \ldots, z^{p_n})
\end{align*}
induces an almost free $\S^1$-action on $\bb S^{2n+1}$. 
We denote the quotient orbifold by 
$$\bb C \bb P^n[p] \coloneqq \bb C \bb P^{n}[p_0,\ldots, p_n] \coloneqq \bb S^{2n+1}/\S^1.$$
We call $\bb C \bb P^n[p_0,\ldots, p_n]$ a \emph{weighted complex projective space}. It is an effective orbifold if and only if $\gcd(p_0,\ldots,p_n) = 1$. The submersion metric has positive sectional curvature.\\
There is slightly different description of weighted complex projective spaces, which allows us to compute its isometry group more easily. This construction is more similar to the definition of quaternionic projective spaces in the next subsection:\\
Let $W$ be an orthogonal representation of $\S^1$ without trivial irreducible subrepresentations. Then $W = \bigoplus_{i = 0}^r k_i W_{q_i}$ for some $k_i \in \bb N$ and $q_i \in \bb Z$, where $W_{q_i}$ is an irreducible $\S^1$-representation with weight $q_i$ and $k_i$ its multiplicity in $W$. We set
$$\bb C \bb P(W) \coloneqq \bb S(W)/\S^1$$
The action of $\S^1$ on the unit sphere $\bb S(W)$ is almost free and hence $\bb C \bb P(W)$ is an orbifold. The principal isotropy group of the $\S^1$-action is trivial if and only if $\gcd(q_0, \ldots, q_r) =1$.  \\
Given such a representation $W = \bigoplus_{i = 0}^r k_i W_{q_i}$, we set $p_{0} = \ldots = p_{k_0-1} = q_0$, $p_{k_0} = \ldots = p_{k_0+k_1 -1} = q_1$, $\ldots$, $p_{k_0 + \ldots + k_{r-1}} = \ldots = p_{k_0 + \ldots + k_r-1} = p_n=q_r$. Then $\bb C \bb P(W) \cong \bb C\bb P^n[p_0,\ldots , p_n]$. Note that $\bb C \bb P^n = \cp^n[1, \ldots, 1] = \bb C \bb P[nW_1]$.
\\
Let $n+1 = \sum_i k_i$. We set $q$ to be the vector consisting of the $q_i$ counted with their multiplicity. We denote by $\S^1_q$ the embedding of $\S^1$ into the orthogonal group $\O(W) \cong \O(n)$. Then the isometry group of $\bb C \bb P(W)$ is given by $$N_{\O(n)}(\S^1_q)/\S^1_q = \frac{\U(k_0) \times \ldots \times \U(k_r) \rtimes \bb Z_2}{\S^1_q},$$ where the $\bb Z_2$-factor is induced by complex conjugation.

\subsubsection{Weighted quaternionic projective spaces.}\label{subsection:wqtproj}
Let $V_m$ be the complex vector space of complex homogeneous polynomials of degree $m$ in two variables $z_1$ and $z_2$. Hence
$$V_m = \langle z_1^m, z_1^{m-1}z_2, \ldots, z_1z_2^{m-1}, z_2^m\rangle_{\bb C} \text{ and } \dim_{\bb C}V_m = m+1$$
There is a natural action of $\SU(2)$ on $V_m$ given by $$A \ast P(z_1,z_2) = P((z_1,z_2) \cdot A^{-1}) \text{ for } P \in V_m, A\in \SU(2)$$
This representation of $\SU(2)$ is irreducible (see \cite{Brocker2010-wu})  and identifies $V_m$ with the complex highest weight representation of weight $\frac{m}{2}$. If $m=2n$ is even, then $V_m$ is of \textit{real type}, i.e. it has a real structure, a conjugate linear map $J \colon V_{2n} \to V_{2n}$ with $J^2 = \Id$. Let $V^\pm_{2n}$ be the $\pm1$ Eigenspace of $J$. These are isomorphic $\SU(2)$-invariant real subspaces of $V_{2n}$.  If $m = 2n+1$ is odd, then this representation is of \textit{quaternionic type}, i.e. there exists a quaternionic structure, a conjugate linear map $J \colon V_{2n+1} \to V_{2n+1}$ with $J^2 = -\Id$. This turns $V_{2n+1}$ into a quaternionic vector space. Any irreducible real representation of $\SU(2)$ is either isomorphic to $V_{2n}^+$ or to $V_{2n+1}$, seen as a real vector space,  giving the real irreducible representations corresponding to highest weight $\frac{m}{2}.$ By Schur's Lemma we have
\begin{align*}
    \End_{\SU(2)}(V_{2n}^+, \bb R) \cong \bb R\\
    \End_{\SU(2)}(V_{2n+1}, \bb R) \cong \bb H
\end{align*}
We will now compute the isotropy groups of the $\SU(2)$-action. Note that any element of $\SU(2)$ is conjugate to an element of the form:
$$ e(t) = \begin{pmatrix}
    e^{it} & 0\\
    0& e^{-it}
\end{pmatrix} \in \U(1) \subset \SU(2)$$
We first compute the isotropy group of the base elements:
$$e(t) \ast z_1^{m-k}z_2^k = e^{i(m-2k)t} z_1^{m-k}z_2^{k}$$
If $m = 2n$ is even, then the element $z_1^nz_2^n$ has isotropy group $\S^1$, so the action on the unit sphere does not give rise to an orbifold quotient. If $m = 2n+1$ is odd, then $-\Id$ always acts as $-\Id \ast P = -P$. Hence it cannot be contained in any isotropy group, so all isotropy groups must be cyclic with odd order \cite[~Chapter 19]{Armstrong2010-nf}. Explicitly one can compute
$$\SU(2)_{z_1^{m-k}z_2^k} = \bb Z_{m-2k}.$$
Any $\SU(2)$-representation $V$ of quaternionic type gives rise to an action on the unit sphere whose quotient is an orbifold with cyclic isotropy groups of odd order. More precisely, let $V = \bigoplus_{i = 0}^r k_{i} V_{2n_i+1}$ then $$\bb H\bb P(V) \coloneqq \bb S(V)/\SU(2)$$
We call $\bb H\bb P(V)$ a \emph{weighted quaternionic projective space}. Note that $\bb H\bb P^n = \bb H\bb P(n V_1)$. Note furthermore that there is an action of $\Sp(k_0)\times \ldots \times \Sp(k_r)$ acting on $V = \bigoplus_{i = 0}^r k_i V_{2n_i+1}$ commuting with the $\SU(2)$-action, given by
\begin{align*}
    (A_0, \ldots, A_r) \ast (v_1^0, \ldots, v_{k_0}^0,\ldots , v_1^r, \ldots, v_{k_r}^r) \\= ((v_1^0, \ldots, v_{k_0}^0) \cdot A_0^{-1}, \ldots, (v_1^r, \ldots, v_{k_r}^r) A_r^{-1})
\end{align*}
This induces an $\Sp(k_0) \times\ldots \times \Sp(k_r)$-action on $\bb H \bb P(V)$. Hence the isometry group of $\bb H \bb P(V)$ is given by
$$N_{\O(n)}(\SU(2))/\SU(2) = \frac{\Sp(k_0) \times\ldots \times \Sp(k_r)}{\bb Z_2}$$
We collect the isometry groups of the simply connected weighted projective spaces and the CROSSes in Table \ref{tab:weightedcrosses}.
{
\renewcommand{\arraystretch}{1.3}
\begin{table}[h!]
    \centering
    \begin{tabular}{|c|c|c|c|}
        \hline
        $\OO$& $I_\std(\OO)$ & $\dim(\OO)$ & $\rk(I_\std(\OO))$  \\\hline\hline
         $\bb S^n$& $\O(n+1)$ & $n$& $\left\lfloor \frac{n+1}{2}\right\rfloor$ \\\hline
         $\bb C \bb P(W)$ & $\frac{\U(k_0) \times \ldots \times \U(k_r) \rtimes \bb Z_2}{\S^1_q}$ & $2\sum_{i = 0}^r k_i-2$& $\sum_{i = 0}^r k_i -1 $\\
         $ W = \bigoplus_{i = 0}^r k_iV_{q_i}$ & & & $= \frac{\dim(\OO)}{2} $\\\hline
         $\bb H \bb P (V)$ & $\frac{\Sp(k_0) \times \ldots \times \Sp(k_r)}{\bb Z_2}$ & $4\sum_{i = 0}^r k_i \cdot (n_i +1) -4$& $\sum_{i=0}^{r} k_i$\\
         $V = \bigoplus_{i = 0}^r k_iV_{2n_i+1}$ &&&\\\hline
         ${\rm Ca} \bb P^2$ & $F_4$ & $16$& $4$\\ \hline
    \end{tabular}
    \caption{Isometry groups of simply connected weighted projective spaces and CROSSes.}
    \label{tab:weightedcrosses}
\end{table}
}

\subsection{Fixed point homogeneous actions on weighted projective spaces}
\subsubsection{Fixed point homogeneous actions on  weighted complex projective spaces}
 Let $ p = (p_0, \ldots,p_n) \in \bb Z^{n+1}$ with all $p_i \neq 0$, and consider the weighted projective space 
 $$\bb C \bb P ^n[p] = \bb C \bb P^n[p_0, \ldots, p_n]$$

 Note that $\bb C \bb P^n[p]$ is not necessarily effective. We are interested in effective fixed point homogeneous $\S^1$-actions on $\bb C \bb P^n[p]$.
For this let $0 \neq q = (q_0,\ldots,q_n) \in \bb Z^{n+1}$ with $\gcd(q_0,\ldots, q_n) = 1$. Then $\S^1$ acts effectively on $\bb S^{2n+1}$ via $\rho_q$ and the action commutes with the action induced by $\rho_p$, and hence induces an action of $\S^1$ on $\bb C\bb P^n[p]$. Furthermore we assume that $q$ and $p$ are linearly independent, since otherwise the induces action by $\rho_q$ would be trivial. We denote by $\phi \coloneqq\rho_p\cdot \rho_q$ an action of $\T^2 = \S^1 \times \S^1$ on $\bb S^{2n+1}$. This induces a linear map
\begin{align*}
    \phi_\ast \colon \bb R^2 \cong \mathfrak{t}^2 &\to  \mathfrak{t}^{n+1} \cong \bb R^{n+1} ,\\
    \phi_\ast &= \begin{pmatrix}
        p_0 & q_0\\
        \vdots&\vdots\\
        p_n &q_n
    \end{pmatrix}
\end{align*}
By a straightforward computation, the principal isotropy of the $\T^2$-action is given by
\begin{align*}
    H &= \lbrace (x,y) \in \bb R^2 \setvert \phi_\ast(x,y) \in \bb Z^{n+1} \rbrace /\bb Z^2\\
    & = \bb Z_{g(p,q)} \oplus \bb Z_{\frac{N(p,q)}{g(p,q)}}.
\end{align*}
Here $g(p,q) = \gcd(p_0,\ldots,p_n,q_0,\ldots,q_n) = 1$, since $\gcd(q) = 1$ and $N(p,q)$ denotes the $\gcd$ of the $2$-minors of $\phi_\ast$, i.e. 
$$N(p,q) = \gcd\lbrace p_iq_j -q_jp_i \setvert i \neq j\rbrace$$
In order for the $\rho_q$-action to be fixed point homogeneous, it needs to fix a codimension two subspace. Therefore, let $(q_1, \ldots,q_n)$ and $(p_1, \ldots, p_n)$ be linearly dependent, i.e. there exists an $a \in \bb Z$, such that $q_i  = a \cdot \frac{p_i}{g}$ with $g = \gcd(p_1,\ldots, p_n)$. Note that in this case $q_0 \neq  a \cdot \frac{p_0}{g}$.  Therefore 
$$p_iq_j- p_jq_i = p_i \frac{p_j}{g}a - p_j \frac{p_i}{g}a = 0$$ and $$ p_iq_0 - p_0q_i = p_i q_0 - p_0 \frac{p_i}{g}a = \frac{p_i}{g}(gq_0 - ap_0).$$
Hence $N(p,q) = q_0g - ap_0$ and $H = \bb Z_{q_0g - a p_0}$. Let $d = \gcd(p)$, then the projection of $H$ to the first $\S^1$ factor has kernel $\bb Z_d$. Therefore, $\rho_q$ acts effectively on $\bb C \bb P^n[p]$, if and only if $\frac1d(q_0g - a p_0) = 1$. This means $q_0 = \frac 1g (ap_0 +d)$. Note that for a given $p$ this is always realizable by B\'ezout's Lemma: Let $x,y \in \bb Z$, such that $xp_0 +y g = d$, then we set $q_0 \coloneqq y$ and $a \coloneqq -x$. Furthermore, we can always pick $0 \le a < g$ and there are $d$ choices for $a$ in this range. 

Since the suborbifold $\lbrace [0,a_1,\ldots,a_n] \in \bb C \bb P[p_0,\ldots p_n]\rbrace\cong \bb C \bb P^{n-1}[p_1,\ldots,p_n]$ is fixed by $\rho_q$ and has codimension $2$, the action is fixed point homogeneous. Note that $\S^1$ acts in the first coordinate of $\cp [p_0,\ldots,p_n]$, but cannot be lifted effectively to an action in the first coordinate of $\bb S^{2n+1}$.

The generic local group of $\bb C \bb P^{n-1}[p_1, \ldots,p_n]$ is $\bb Z_g$. We will now describe a tubular neighbourhood around the isolated fixed point $[1,0,\ldots,0]$, which turns $\bb C \bb P^n[p]$ as a double disk bundle: We assume that $\bb C \bb P^n[p]$ is effective, since we only need this for our proof. Hence $d = 1$. The $\rho_q$ action fixes the point $[1, 0,\ldots,0]$. We will now compute a lift $K$ of the isotropy group: $1 \to \bb Z_{p_0} \to K \to \S^1 \to 1$. Since $\gcd(p_0,q_0) = 1$, the isotropy group of the $\T^2$-action is given by
\begin{align*}
    K = \T^2_{[1,0,\ldots,0]} &= \lbrace (z,w) \in \T^2 \setvert z^{q_0} w^{p_0} = 1 \rbrace\\
    &= \lbrace (z^{p_0}, \bar z^{q_0}) \setvert z \in \S^1\rbrace   \cong \S^1
\end{align*}
Now $ K = \S^1$ acts on $\lbrace (0,a_1,\ldots,a_n)\in \bb S^{2n+1}\rbrace$ by 
\begin{align*}
    (\bar z^{p_0},  z^{q_0}) \ast (0,a_1,\ldots,a_n) &= (0, z^{(-p_0 \cdot q_1 + q_0\cdot p_1)}a_1, \ldots, z^{(-p_0 \cdot q_n + q_0\cdot p_n)}a_n)\\
    & = (0, z^{\frac{p_1}{g}(gq_0-ap_0)}a_1, \ldots, z^{\frac{p_n}{g}(gq_0 - ap_0)}a_n)\\
    &= (0, z^{\frac{p_1}{g}}a_1, \ldots, z^{\frac{p_n}{g}}a_n) 
\end{align*}
We see immediately that $(\bb C\bb P^{n-1}[p_1, \ldots, p_n])_\eff = \bb C\bb P^{n-1}\left[\frac{p_1}{g}, \ldots, \frac{p_n}{g}\right]$, i.e. we recover the effective model of the fixed point component. Hence a tubular neighbourhood around $[1, 0, \ldots,0]$ is isomorphic to 
$\S^1 \times_{K}\bb D^{2n}$, with
$$z \ast (u,a_1, \ldots, a_n) = (\bar z^{p_0}u,  z^{\frac{p_1}{g}}a_1, \ldots, z^{\frac{p_n}{g}}a_n)$$

\begin{exa} \label{exa:fp:S1}
    The fixed point homogeneous $\S^1$-actions on $\cp [p_0,\ldots,p_n]$ are given, after possibly reordering the $p_i$, by 
    \begin{align*}
        z \ast [a_0,\ldots, a_n] = [z^{\frac{1}{g}(ap_0 +d)}a_0, z^{\frac{p_1}{g}a}a_1,\ldots , z^{\frac{p_n}{g}a}a_n ]
    \end{align*}
    Here $g = \gcd(p_1,\ldots,p_n)$, $d = \gcd(p_0, \ldots,p_n)$ and $0\le a <g$ such that $\frac{1}{g}(ap_0 +d)$ is an integer. There are always exactly $d$ choices.\\
    Note that these describe the possible effective lifts to $\S^{n+1}$. On $\cp[p_0,\ldots,p_n]$ these are all equivalent to the coordinate action $z\ast[a_0:\ldots:a_n]=[z^{\frac{d}{g}}a_0:a_1:\ldots:a_n]$. Note that this action does not lift to a coordinate action on $S^{n+1}$ unless $d=g$ and one can take $a=0$.
\end{exa}

\noindent
We will now give a description of $\SU(k)$-actions on weighted projective spaces: \begin{align*}
    \SU(k) &\acton \bb C\bb P^{k+n}[p,\ldots,p,q_0,\ldots,q_n]\\
    A \ast[a_0,\ldots, a_{k-1}, a_{k}, \ldots, a_{k+n}] &= [ (a_0,\ldots, a_{k-1}) \cdot A^\ast,a_{k}, \ldots, a_{k+n}]
\end{align*}
With $\gcd(p,q_0,\ldots,q_n) = 1$. Let $g = \gcd(q_0, \ldots,q_n)$. Note that the kernel of this action is given by $\bb Z_{\gcd(k, g)}$. The fixed point set is given by $$\bb C\bb P^{n}[q_0,\ldots,q_n] = \lbrace [0,\ldots,0,a_0 \ldots, a_n] \in \bb C\bb P^{k+n} [p,\ldots,p, q_0,\ldots, q_n]\rbrace $$
The effective structure is given by $\bb C\bb P^{n}[q_0^\prime,\ldots,q_n^\prime]$, where $q_i^\prime = \frac{q_i}{g}$. The orbit $\SU(k) \cdot[1,0,\ldots,0]$ has isotropy group $\U(k-1) \subset \SU(k)$ and is diffeomorphic to $\bb C\bb P^{n}$, embedded with local group $\bb Z_p$. To compute the local lift $1 \to \bb Z_p \to K_p \to \U(k) \to 1$ we compute the isotropy of the orbit of the lifted action of $\SU(k) \times \S^1$ on $\bb S^{2(k+n)+1}$.  This is given by $K_p=\lbrace (\diag(z^p,B), z) \vert \det B = \bar z^p\rbrace \cong \SU(k-1) \times_{\bb Z_{k-1}}\S^1$, where $\bb Z_{k-1} \to \bb Z_{\frac {k-1} {\gcd(k-1,p)}} \subset \SU(k-1)$.  $K_p$ projects to $\U(k)$ with kernel $\bb Z_p$. $L = \lbrace \diag (\bar w, B) \setvert \det B  = w, w \in \bb Z_g\rbrace \cong \SU(k-1) \times_{\bb Z_{\gcd(g, k-1)}} \bb Z_g $ and $H_p = \lbrace (\diag (\bar z ^p, B),z)\setvert \det B = z^p, z \in \bb Z_g\rbrace \cong \SU(k-1) \times_{\bb Z_{\gcd(p,g)}} \bb Z_g$. The normal bundle of $\SU(k)\cdot[1,0\ldots,0]$ is diffeomorphic to $$\SU(k) \times_{\SU(k-1)\times_{\bb Z_{k-1}}\S^1 } \bb D^{2n+2}.$$
If $k = 2m$, then $\Sp(m) \subseteq \SU(k)$ acts transitively on the normal disk to $\bb C\bb P^n[q_0,\ldots,q_n]$. The kernel of the $\Sp(m)$-action on $\bb C\bb P^{2m+n}[p,\ldots,p,q_0,\ldots,q_n]$ is given by $\bb Z_{\gcd(2,g)}$. $G_p = \lbrace \diag(w,\bar w, B)\setvert w \in \S^1,\, B\in \Sp(k-1)\rbrace$. And $K_p= \lbrace (\diag(z^p,\bar z^p,B),z)\setvert z \in \S^1,\, B \in \Sp(k-1)\rbrace \cong \Sp(k-1) \times \S^1$. $L = \lbrace \diag(w, \bar w, B)\setvert B \in \Sp(k-1),\, w \in \bb Z_g\rbrace$ and $H_p = \lbrace (\diag (z^p,\bar z^p, B), z) \setvert z \in \bb Z_g, B \in \Sp(k-1)\rbrace \cong \Sp(k-1) \times \bb Z_g$.

\begin{exa}\label{exa:fp:cp}
\mbox{}
\begin{enumerate}
    \item $\SU(k)$ acts fixed point homogeneous on $\OO = \cp^{k+n}[p,\ldots, p, q_0, \ldots,q_n ]$ by
    \begin{align*}
        A \ast [a_0,\ldots, a_{k+n}] = [(a_0,\ldots,a_{k-1})\cdot A^\ast,a_k, \ldots, a_{k+n}].
    \end{align*}
    \item $\Sp(m)$ acts through the embedding $\Sp(m) \to \SU(2m)$ and the action above on $\OO$.
    \item $\U(k)$ and  $\U(1) \times \Sp(m)$ act fixed point homogeneous on $\OO$, such that the $\SU(k)$ and $\Sp(m)$ act as above and the $\U(1)$ factor is in $Z_{I_\std(\OO)}(\SU(k))$.
\end{enumerate}
    
\end{exa}

\subsubsection{Fixed point homogenous actions on  quaternionic wighted  projective spaces} \label{subsection:quaternionic}

The only fixed point homogeneous actions are the following:\\ Let $V = \bigoplus_{i = 0}^r k_iV_{2n_i+1}$ be a quaternionic representation with $n_0 = 0$ and $1  \le k_0^\prime \le k_0$. Then $\Sp(k_0^\prime)$ acts fixed point homogenous on $\bb H \bb P(V)$, since it fixes $\bb S(V) \cap ((k_0 - k_0^\prime)V_1 \oplus \bigoplus_{i = 1} ^{r} k_i V_{2n_i+1})$ and obviously acts transitively on the normal sphere embedded in $k_0^\prime V_1 \cong \bb H^{k_0^\prime}$. The isotropy group of the $\Sp(k_0^\prime)$-action at $p =[1,0, \ldots, 0]$ is given by $\Sp(1) \times\Sp(k_0^\prime -1)$. Furthermore the $\Sp(k_0^\prime-1)$ factor acts trivially on the normal sphere to $\Sp(k_0^\prime)\cdot p$ and the $\Sp(1)$-factor acts with representation $(k_0 - k_0^\prime)V_1 \oplus \bigoplus_{i = 1}^rk_iV_{2n_1+1}$. 

\begin{exa}\label{exa:fphquat}
    Let $V = \bigoplus_{i = 0}^r k_i V_{2n_i+1}$ with $n_0 = 0$. Then $\Sp(k_0^\prime)$ with $1 \le k_0^\prime \le k_0$ acts fixed point homogeneous on $\bb H\bb P(V)$
\end{exa}

We conclude this section with a table of every fixed point homogeneous action on the sphere, the Cayley plane and the weighted projective spaces. 
{
\renewcommand{\arraystretch}{1.3}
\begin{table}[h!]
    \centering
    \begin{tabular}{|c|c|c|}
    \hline
        Space & Possible Groups & Group Action \\\hline
        \hline
         $\bb S^n$& $\SO(k), \SU(k), \U(k), \Sp(k), \Sp(k)\U(1),$ & See Ex. \ref{exa:fphspheres}\\
         & $\Sp(k)\Sp(1), \Spin(9), \Spin(7), \gG_2$ & \\\hline
         $\Ca \bb P^2$ & $\Spin(9)$ & See Ex. \ref{exa:fphcayley} \\\hline
         $\cp[p_0, \ldots, p_n]$ & $\U(k), \SU(k), k\le \max \text{mult}(p_i)$ & See Ex. \ref{exa:fp:S1} and \ref{exa:fp:cp}\\
          & $\Sp(k), \Sp(k)\cdot \U(1), 2k\le \max \text{mult}(p_i)$& \\\hline
         $\hp(V)$ & $\Sp(k), k\le\text{mult(standard }S^3\acton\HH\text{ in }V)$ & See Ex. \ref{exa:fphquat}\\\hline
    \end{tabular}
    
    \caption{The spaces of Theorem \ref{main:A}}
    \label{tab:examples}
\end{table}
}


\section{Fixed point homogeneous Orbifolds}\label{sec:effective}
A fixed point homogeneous orbifold is an orbifold $\OO$ together with a $G$-action on $\OO$ such that there is a non-empty fixed point component  $F\subseteq\Fix(\OO,G)$ such that $G$ acts transitively on the space of normal directions of $F$. Fixed point homogeneous actions on positively curved manifolds have been classified in \cite{Grove1997-xz}, the results have been partially extended in \cite{Spindeler2020-vj} to non-negative curvature. The proof is roughly as follows: Let $M$ be a fixed point homogeneous manifold and $F$ a fixed point component, such that $G$ acts transitively on the normal sphere. Let $\pi\colon M \to M/G$ the quotient map and $\bar F = \pi (F)$. Then $\bar F$ corresponds to a boundary component of $M/G$. Furthermore $M/G$ is positively curved. By the Soul theorem for Alexandrov spaces the distance map $d_{\bar F}$ is concave and hence there is a point in $M/G$ at maximal distance corresponding to a $G$-orbit $G\cdot p$ in $M$. Therefore $M$ is diffeomorphic to a double disc bundle
\begin{align*}
    M = \mathcal{V}F \cup_\partial \mathcal{V} G\cdot p
\end{align*}
Furthermore all orbits in $M \bs( F\cup G\cdot p)$ are principal and diffeomorphic to $\mathcal{V}^1_q(F) \cong \bb S^k \cong G/H$. Homogeneous spheres are classified, the possibilities were listed in Table \ref{tab:fixhom:homospheres} above.

It is now easy from the structure of the double disc bundle and the table to construct equivariant diffeomorphism between compact rank one symmetric spaces with fixed point homogeneous actions and the manifold $M$ in the corresponding case for $G/H$. 
\\
We will generalize this construction to orbifolds. Note, that the space of normal directions to $F$ is in general a manifold space form $\bb S^k/\Gamma$ and might be singular as a suborbifold, when embedded in $\OO$.\\
In particular, our main result is:
\begin{theorem}\label{main:fphom}
        Let $\OO$ be a simply connected, fixed point homogeneous and effective $G$-orbifold with positive sectional curvature, such that $G$ acts almost effectively. Then up to a lift of the $G$-action $\OO$ is either $G$-equivariantly diffeomorphic to one of the following cases:
        \begin{enumerate}[label = \alph*)]
            \item A sphere $\bb S^{l+1+m}$, such that $G$ acts transitively on $\bb S^l$, with action
            $$g \ast (a_0,\ldots, a_{l+m}) = (g\ast(a_0,\ldots,a_l), a_{l+1}, \ldots, a_{l+m})$$
            \item One of the following cases, with the numbering according to Table \ref{tab:fixhom:homospheres}:
            \begin{enumerate}[label = \arabic*., start = 2]
            \item $n \ge 2$, $\bb C \bb P(W)$ for $W =  \bigoplus_{i = 0}^r k_i W_{q_i}$ for some $k_i \ge n+1$, $q_i \in \bb Z$ with $\gcd(q_0, \ldots q_r) = 1$ an the action from example \ref{exa:fp:cp}.
            \item[2.1.] $n \ge 0$, $\bb C \bb P(W)$ for $W =  \bigoplus_{i = 0}^r k_i W_{q_i}$ for some $k_i \ge n+1$, $q_i \in \bb Z$ with $\gcd(q_0, \ldots q_r) = 1$ and the actions from examples \ref{exa:fp:S1} and \ref{exa:fp:cp}.
            \item $n \ge 0$, $\bb C \bb P(W)$ for $W =  \bigoplus_{i = 0}^r k_i W_{q_i}$ for some $k_i \ge 2n+2$, $q_i \in \bb Z$ with $\gcd(q_0, \ldots q_r) = 1$ and the action from example \ref{exa:fp:cp}, or $\bb H \bb P(V)$ for $V = \bigoplus_{i=0}^{r} k_iV_{2n_i+1}$, such that $n_0 = 0$ and $k_0 \ge n+1$ and the action from example \ref{exa:fphquat}. 
            \item[3.1.] $n \ge 0$,            $\bb C \bb P(W)$ for $W =  \bigoplus_{i = 0}^r k_i W_{q_i}$ for some $k_i \ge 2n+2$, $q_i \in \bb Z$ with $\gcd(q_0, \ldots q_r) = 1$ and the action from example \ref{exa:fp:cp}.
            \item ${\rm Ca}\bb P^2$ with the subaction of $\Spin(9) \subseteq F_4$.
        \end{enumerate}
        \end{enumerate}
    
\end{theorem}
\begin{rem}\mbox{}
\begin{enumerate}
    \item  In theorem \ref{main:fphom} it suffices  to assume that $\vert \OO \vert/G$ is positively curved in Alexandrov sense.
    \item  The obvious difference to the case where $\OO$ is a manifold in the classification of Grove and Searle \cite{Grove1997-xz} is the occurrence of the weighted complex and quaternionic projective spaces. It is interesting to note that all weighted complex projective spaces show up in our classification, but not all weighted quaternionic projective spaces. Only those, where $\Sp(1) \acton V$ contains a subrepresentation equivalent to the Hopf-action. Furthermore all of these quaternionic projective spaces are manifolds away from the fixed point set. 
\end{enumerate}

\end{rem}
We will postpone the proof of theorem \ref{main:fphom} and first derive the case where $\OO$ has a fundamental group.
\begin{lem}\label{lem:fixhom:linear}
    Let $L$ be a Lie group acting effectively on a simply connected fixed point homogeneous $G$-orbifold $\OO$ by $G$-equivariant diffeomorphisms. Then $L$ must be a linear action and contained in the centralizer of $G$ in the isometry group of the standard metric on $\OO$.
\end{lem}
\begin{proof}
    Since $L$ acts by $G$-equivariant diffeomorphisms, we have, that the actions of $G$ and $L$ commute. Therefore $l \in L$ maps a $G$-orbit to one of the same isotropy type. \\
    We first consider the case, where $G$ acts by cohomogeneity one, hence $\OO = D^{l+1}(p) \cup D^{l+1}(q)$ for two fixed points $p$ and $q$. In this case  either $\OO \cong \bb S^{l+1}$, and $G$ acts transitively on the equator $\bb S^l$ or $\OO \cong \bb C\bb P^1[m,n] $, with $\gcd(m,n) = 1$.  There are two options: $L$ fixes $p$ and therefore acts with the isotropy representation and is therefore linear, or $L \neq L_p$. Therefore any element in $L \setminus L_p$ interchanges $p$ and $q$. This means that $p$ and $q$ have the same local groups and therefore in the latter case $m = n = 1$, which coincides with the first one if $l =1$. Let $A = \lbrace x \in \OO \setvert d(x,p) = d(x,q)\rbrace$. Since $G$ acts by cohomogeneity one, we have that $A =G\cdot z$ is a single orbit. Furthermore it is easy to see, that $L$ fixes $A$. Now $L$ acts smoothly on $D(G/H) = (-1,1) \times G/H$. This induces an action of $L$ on the second factor which fixes the origin. Hence $L$ acts by reflections. Furthermore $L$ acts on $G/H$ by linear maps (see \cite[Sublemma~2.6]{Grove1997-xz}). Hence, the action is linear in total.\\
    In the remaining cases, the fixed point component is infinite and hence will always be fixed by the $L$ action. Therefore $L$ respects the structure of the double disc bundle, ie. $L$ also fixes the orbit $G\cdot p$. Again by \cite[Sublemma~2.6]{Grove1997-xz}, we have that the action is linear on $G/G_p$. \\
    The normal exponential map $\exp\colon \V^{\le 1} G\cdot p \to  D(G\cdot p)$ is an $L \times G$-equivariant diffeomorphism. Hence the action of $l \in L$ is linear on the fibers, i.e. the map $l$ maps a slice at $p$ to a slice at $l\ast p$ and this action is equivalent to $dl \colon \bb D_p/\Gamma_p \to \bb D_{l\ast p}/\Gamma_p$. Since the action respects the double disc bundle structure it extends linearly.
\end{proof} 
\begin{cor}\label{fphom:cor:fundamental}
    Let $\OO$ be a fixed point homogeneous and effective $G$-orbifold with positive sectional curvature. Let  $\tilde\OO$ be the universal covering of $\OO$, such that $\OO = \tilde\OO/\Gamma$ for a finite group $\Gamma$, then $\Gamma \subset Z_{{\rm Isom}_{\rm std}(\tilde\OO)}(G)$, where ${\rm Isom}_{\rm std}(\tilde\OO)$ is the isometry group of the standard metric.\\
    Furthermore $Z_{{\rm Isom}_{\rm std}(\tilde\OO)}(G)$ is given by:
    \begin{enumerate}[label = \alph*)]
        \item $\Isom_\std(\tilde \OO) \cong \O(l+1+m)$
        \begin{multicols}{2}
            \begin{enumerate}[label = \arabic*.]
            \item    $\bb Z_2 \times \O(m)$, $l = n$
            \item $\U(1) \times \O(m)$, $l = 2n+1$, $n \ge 2$
            \item[2.1] $\U(1) \times \O(m)$, $l = 2n+1$
            \item $\Sp(1) \times \O(m)$, $l = 4n+3$
            \item[3.1] $\U(1) \times \O(m)$, $l = 4n+3$
            \item[3.2] $\bb Z_2 \times \O(m)$
            \item $\bb Z_2 \times \O(m)$
            \item $\bb Z_2 \times \O(m)$
            \item $\bb Z_2 \times \O(m)$
        \end{enumerate}
        \end{multicols}
        \item 
        \begin{enumerate}[label = \arabic*., start = 2]
            \item $(\U(k_0) \times \ldots \times \U(k_{i-1}) \times \U(k_i-n-1) \times \U(k_{i+1}) \times \ldots \times \U(k_r))/\bb Z_{q_i}$
            \item[2.1] $(\U(k_0) \times \ldots \times \U(k_{i-1}) \times \U(k_i-n-1) \times \U(k_{i+1}) \times \ldots \times \U(k_r))/\bb Z_{q_i}$ 
            \item $(\U(k_0) \times \ldots \times \U(k_{i-1}) \times \U(k_i-2n-2) \times \U(k_{i+1}) \times \ldots \times \U(k_r))/\bb Z_{q_i}$ for $n \ge 1$,\\
            $(\U(k_0) \times \ldots \times \U(k_{i-1}) \times \U(k_i-2) \times \U(k_{i+1}) \times \ldots \times \U(k_r))/\bb Z_{q_i}\rtimes \bb Z_2$ for $n = 0$
            or \\
             $(\Sp(k_0 - n-1) \times \Sp(k_2) \times \ldots \times \Sp(k_r))/\bb Z_2$
             \item[3.1]  $(\U(k_0) \times \ldots \times \U(k_{i-1}) \times \U(k_i-2n-2) \times \U(k_{i+1}) \times \ldots \times \U(k_r))/\bb Z_{q_i}$
             \item $\bb Z_2$
        \end{enumerate}
    \end{enumerate}
    
\end{cor}
\begin{proof}
    This easily follows from Theorem~\ref{main:fphom} and Lemma~\ref{lem:fixhom:linear} by computing the centralizers in the isometry groups from Table~\ref{tab:weightedcrosses}. 
\end{proof}

We obtain the following corollary
\begin{cor}\label{cor:fphom:maxsymmetry}
    Let $\OO^n$ be a positively curved orbifold, then:
    \begin{enumerate}
        \item The symmetry rank of $\OO$ is bounded above by $\lfloor \frac{n-1}{2}\rfloor$.
        \item If equality holds in 1., then the universal cover of $\OO$ is diffeomorphic to a sphere or a a weighted complex projective space. In fact the diffeomorphism is equivariant by the maximal torus in $I(\OO)$
    \end{enumerate}
\end{cor}
\begin{proof}
    1. was shown in \cite{HarSea17}. For 2. let $\T^k$ be the maximal torus of the isometry group of $\OO$. By \cite[Theorem~6.1]{HarSea17} there is an $\S^1\subseteq \T^k$, such that $\S^1$ acts fixed point homogeneous on $\OO$. Now the result follows from Theorem \ref{main:fphom} and the equivariance of the diffeomorphism from Lemma \ref{lem:fixhom:linear}.
\end{proof}
To start the proof of theorem \ref{main:fphom}, we need the following
\begin{lem}\label{lem:fixedpoints}
    Let $\OO$ be a Riemannian orbifold and $G$ a Lie Group acting by isometries. Then the following hold:
    \begin{enumerate}
        \item The fixed point set is a totally geodesic closed subset.
        \item If for all $p \in \OO ^G$ we have $\gcd(\vert G/G_0\vert, \vert \Gamma_p \vert)  = 1$, any component of $\OO^G$ is a totally geodesic  suborbifold of $\OO$.
    \end{enumerate}
\end{lem}
\begin{proof}\mbox{}
    \begin{enumerate}
        \item This is a simple fact from Alexandrov geometry.
        \item Let $F$ be a component of $\OO^G$ and $p \in F$. Let $U \subseteq \abs{\OO}$ be a $G$-invariant neighbourhood of $p$ and $(\hat U,\Gamma_p)$ a local model of $\OO$, such that $0 \in \hat U$ and $p = [0] \in \hat U/\Gamma_p$. There exists a lift of $G$ given by $1 \to \Gamma_p \to K_p \to G \to 1$, such that $K_p$ extends the action of $\Gamma_p$ on $\hat U$. Now consider the tangent space $T_0\hat U$. The tangent cone at $p$ is given by: $C_p\OO \cong T_0\hat U/\Gamma_p$. We are interested in $C_p(\OO)^G$. The lift of this space to $T_0\hat U$ is given by
        $$A = \left\lbrace v \in T_0\hat U \setvert \forall k \in K_p, \exists \, \gamma \in \Gamma_p \colon kv = \gamma v\right\rbrace$$
        Obviously $0 \in A$. We claim that $A$ is a linear subspace of $T_0\hat U$\\ 
        First we assume, that $G$ is connected. If $A = \lbrace 0 \rbrace$, we are done. Hence let $v \in A \setminus \lbrace 0\rbrace$. Since $G$ is connected, we have $\Gamma_p\cdot (K_p)_0 = K_p$. Let $K_0, \ldots,K_r$, be the components of $K_p$ with $K_0 = (K_p)_0$. Then there exists $\gamma_0, \ldots, \gamma_r\in \Gamma$, such that $\gamma_i^{-1}\cdot K_i = K_0$. We choose $\gamma_0 = \Id$. Since the action is smooth and $\Gamma_p$ is discrete, we have $k v = v$ for all $k \in K_0$ and $v \in A$, because of $\Id(v) =v$. Therefore for any $k \in K_i$ there exists $k^\prime\in K_0$, such that, we have $k v = \gamma_ik^\prime v = \gamma_i v$ and hence $\gamma_i^{-1} k v = v$ for each $i = 1, \ldots,r$. Thus $A = \Fix(K_0)$ and hence is a subspace of $T_0\hat U$. we denote the kernel of the $\Gamma_p$-action on $A$ by $\Gamma_A$. By applying the exponential map to the subspace $A$, we get an orbifold chart of $F$, with orbifold group given by $\Gamma_p/\Gamma_A$.\\
        Now let $G$ be finite. Since the order and the index of $\Gamma_p$ in $K_p$ are relatively prime by the Schur-Zassenhaus theorem (\cite[Theorem~7.41]{Rotman2012-zb})  there exists a complement $L_p$ of $\Gamma_p$ in $K_p$, i.e. $K_p = L_p \cdot \Gamma_p$ and $L_p \cong G$. We claim, that the subspace $(T_0\hat U)^{L_p} \subseteq A$ maps surjectively on $C_p(\OO)^G$. Let $[x] \in C_p(\OO)^G$. Then the stabilizer $(K_p)_x$ of $x$ in $K_p$ maps surjectively on $G$: Let $g \in G$  and $\tilde g\in K_p$ a lift of $g$. Then there exists $\gamma\in \Gamma$ with $gx = \gamma x$. Hence, $\gamma^{-1} gx = x$ and therefore the $(K_p)_x \to G$ is surjective, with kernel $(\Gamma_p)_x$. Hence again by the Schur-Zassenhaus theorem there is a complement $L_x$ of $(\Gamma_p)_x$ in $(K_p)_x$. $L_x$ is also a complement of $\Gamma_p$ in $K_p$: Suppose there exists $l \in L_x \cap \Gamma_p$. Since $L_x \subseteq (K_p)_x$, we have that $l \in (\Gamma_p)_x$ and hence $l = e$. By Theorem 7.42 in \cite{Rotman2012-zb} and the Feit-Thompson Theorem  (\cite[Theorem 1]{FeitThompson1962}), we have that $L_x$ and $L_p$ are conjugate by an element $\gamma \in \Gamma_p$. Hence $L_p$ fixes $\gamma^{-1}x$. Since $[\gamma^{-1}x] = [x]$, the claim follows and hence $\OO^G$ is a totally geodesic suborbifold as well.\\
        For the full statement, we apply the second case to $\OO^G = (\OO^{G_0})^{G/G_0}$.
    \end{enumerate}
\end{proof}
\begin{rem}
    There is a counterexample if the condition on the order of $G/G_0$ and $\Gamma_p$ is not fulfilled. Consider $\bb D^2/\bb Z_{2k}$, with a $\bb Z_2$ action a reflection at a plane. Then $(\bb D^2/\bb Z_{2k})^{\bb Z_2}$ is not a suborbifold, since there is no subspace in $\bb D^2$, that maps surjectively to the fixed point set.    
\end{rem}
We start with the following structure result, which is a generalization of the manifold case \cite{Grove1997-xz}:
\begin{theorem}[Structure Theorem]\label{thm:strucure}
Let $\OO$ be a Riemannian orbifold with positive sectional curvature with an (almost) effective isometric fixed-point homogeneous $G$-action and $\OO^G \neq\emptyset$. Let $F$ be a component of $\OO^G$ with maximal dimension, then:
\begin{enumerate}
    \item There is a unique orbit $G\cdot p$ at maximal distance to $F$ (the Soul-orbit). Moreover $(G\cdot p)_{\eff} \cong G/G_p$. 
    \item All $G_p$-orbits on the space  $\V_p^1(G\cdot p) =\bb S^l/\Gamma_p$ of normal directions to $G\cdot p$ are almost principal. Moreover principal orbits are contained in $\OO_{reg}$. The principal orbits are isomorphic to $G_p/H$, where $H$ denotes the principal isotropy group, and $F_{\eff} \cong (\V_p^1(G\cdot p)/G_p)_{\eff} $
    \item There is a $G$-equivariant decomposition of $\OO$ as $$\OO = D(F) \cup_E D(G \cdot p),$$
    where $D(F)$ and $D(G \cdot p)$ denote the normal bundles of $F$ and $G \cdot p$ with common boundary $E$ in $\OO$.
    \item All orbits in $\OO \setminus F \cup G \cdot p$ are almost principal. And the principal orbits are contained in $\OO_{\reg}$ and diffeomorphic to $\V^1_q(F) \cong G/H$ for an intrinsically regular point $q\in F$.
\end{enumerate}
\end{theorem}
\begin{proof}
    We will assume that $G$ is connected. We begin with some preliminaries. Let $\pi \colon \OO \to \OO/G$ be the quotient map and equip the quotient, with the orbit distance metric. Furthermore let $\bar F \coloneqq \pi(F)$.\\
    Claim 1. follows as in \cite{Grove1997-xz} from an easy application of the Soul Theorem for Alexandrov space \cite{Perelman}. \\
    \textbf{Claim:} A small $\epsilon$-neighborhood $B_\epsilon(\bar F)$ in $\OO/G$ is diffeomorphic to an orbifold with boundary $F_{\eff}$ and $\bar F = \vert F_{\eff}\vert$.\\
    To see this let $q \in \abs{F}$ and $U \subseteq\abs{\OO}$ and open subset together with a local model $(\hat U, \Gamma_q)$. Since $p$ is a fixed point of the $G$-action, we get a local lift $1 \to \Gamma_q\to K_q\to G \to1$ acting on $\hat U$. By the proof of Lemma~\ref{lem:fixedpoints} the tangent directions of $F$ correspond to the fixed points of $K_0$ in $T_q\hat U$. Thus we get a splitting $$T_q\hat U \cong T_q\hat U^{K_0} \oplus (T_q\hat U^{K_0})^\perp$$
    The elements in $(T_q\hat U^{K_0})^\perp$ correspond to the normal direction of $F$. Now we first take the quotient by $K_0$ and get a diffeomorphism
    $$T_q\hat U/K_0 \cong T_q\hat U^{K_0} \times [0,1).$$
    $K_q/K_0 \cong \Gamma_q/(\Gamma_q\cap K_0)$ now acts on this quotient, and we can endow the quotient with the local structure of $F_{\eff}\vert_{U}$. With this, the quotient map also becomes smooth. \\
    Together with critical point theory of distance functions on orbifolds (see \cite{KL14}) this shows the remaining claims.
\end{proof}
\begin{lem}[Uniqueness Lemma]\label{lem:unique}
    Let $\OO$ and $\hat \OO$ be two Riemannian $G$-orbifolds with the structure of Theorem \ref{thm:strucure}, i.e. there exist components $F\subseteq \OO^G$ and $\hat F \subseteq \hat \OO$ and orbits $G\cdot p \subseteq \OO$ and $G\cdot \hat p \subseteq \hat \OO$, such that 2. -- 4. hold. Furthermore assume, that the orbibundles $D(G\cdot p) \to G/G_p$ and $D(G \cdot \hat p) \to G/G_{\hat p}$ are isomorphic. Additionally, let one of the following conditions hold:
    \begin{enumerate}
        \item $\codim F = \codim \hat F \ge 3$ or
        \item  $\codim F = \codim \hat F = 2$ and additionally the local groups of (possibly generically singular) manifold points of $F$ and $\hat F$ are isomorphic.
    \end{enumerate}
    Then $\OO$ and $\hat \OO$ are $G$-equivariantly diffeomorphic.
\end{lem}
\begin{proof}
    We will follow the proof of \cite{Grove1997-xz}. We show, that any equivariant isomorphism of  orbi bundles $f \colon D(G\cdot p) \to D(G \cdot \hat p)$ extends to an equivariant diffeomorphism $\OO \to \hat \OO$ under one of the assumptions in the uniqueness lemma. In the first condition note, that $\V^1_q(F) \cong \bb S^k/\Gamma_q$, $\V_{\hat q} (\hat F) \cong \bb S^k/\hat \Gamma_q$ at an intrinsically regular point of $F$, resp. $\hat F$. Since the bundles are isomorphic, we have $\Gamma_q \cong \pi_1(\bb S^k/\Gamma_q) \cong \pi_q(\bb S^k/\hat \Gamma_{\hat q}) \cong \hat \Gamma _{\hat q}$. At this point, we need $k\ge 2$, since $\S^1/\Gamma \cong \S^1$ for any finite group acting almost effectively on $\S^1$.  Hence in both cases we can assume, that the corresponding local groups at intrinsically regular points of the fixed point components are isomorphic. Let $\Gamma$ (resp. $\hat\Gamma$) be the isotropy of an intrinsically regular point of $F$ (resp. $\hat F$). The bundles $D(F) \to F_{\eff}$, resp. $D(\hat F) \to \hat F_{\eff}$ are orbibundles with fiber $\bb D^{k+1}/\Gamma$ and $\bb D^{k+1}/\hat \Gamma$. Since $\Gamma$ and $\hat\Gamma$ are isomorphic, this map has a unique radial extension, which is smooth by sublemma 2.6 in \cite{Grove1997-xz}. Therefore $f$ extends to a diffeomorphism of $\OO$ and $\hat \OO$. 
\end{proof}

From now let $\OO$ be an effective and simply connected orbifold with positive sectional curvature,  such that a compact Lie group $G$ acts fixed point homogeneous and effectively on $\OO$. Let $q\in F$  an intrinsically regular point and $\Gamma_q$ the local group. Then we have $\V_q^{\le1}F \cong \bb D^{l+1}/\Gamma_q$. Since $G$ acts transitively on $\V_q^1F$, it is an orbit and hence a manifold. Therefore, $\Gamma_q$ acts effectively free on $\bb S^l$. We get an extension $1 \to \Gamma_q \to K_q \to G \to 1$, such that $K_q$ acts transitively on $\bb S^l$, with isotropy group $H_q$. The preimage of the principal isotropy $L$ of the $G$ action under the quotient map $K_q \to G$ must be a $\Gamma_q$-extension of $H_q$. We denote by $K$ the unit component of $K_q$ and  by $H$, the unit component of $H_q$. We are now computing through the pairs of $(K,H)$ from the table. Since $\Gamma_q$ acts trivially on $F$, we can embed $\Gamma_q \cdot K = K_q$ into $\O(k+1)$, where $k+1 = \codim F$.

\begin{lem}\label{fixpoint:lem:doubledisc}
    Let $\OO = D^{k+1}(F) \, \cup_E D^{l+1}(S)$ a double orbidisc bundle,  where $E = \partial D^{k+1}(F) = \partial D^{l+1}(S)$. Then
    \begin{enumerate}
        \item $F \hookrightarrow\OO$ is orbi-$l$-connected.
        \item $S \hookrightarrow \OO$ is orbi-$k$-connected.
        \item $E \hookrightarrow \OO$ is orbi-$\min(k,l)$-connected.
    \end{enumerate}
\end{lem}
\begin{proof}
    By picking a suitable Riemannian metric, it is easy to see, that $$\Fr(\OO) = D^{k+1}(\Fr(\OO)\vert_F) \, \cup D^{l+1}(\Fr(\OO)\vert_S).$$
    Hence $\Fr(\OO)$ is as a manifold diffeomorphic to a double disc bundle for which the claim is known (see for example \cite[(1.4)]{gwz}). The result follows from the long exact sequence of homotopy groups for the fiber bundles and the five lemma.
\end{proof}
\begin{cor}
    If $\dim F \ge2$, then $F$ is simply connected.
\end{cor}
\begin{proof}
    By lemma \ref{fixpoint:lem:doubledisc}, the inclusion of $F$ is $l$-connected. Since $\dim F = \dim (\bb S^l/K_p) \le l$ the result follows.
\end{proof}

\begin{cor}\label{fixpoint:lem:connectedgroup}
\mbox{If $k\ge 2$, then $K_p$ is connected.}
\end{cor}

\begin{proof} 
 By the long exact sequence of homotopy groups for the bundle
    $$K_p \to EK_p \times G \to B(G/K_p)$$
    We have that $\pi_0(K_p) = 0$, if $\pi_1^\orb(G/K_p) = 0$. Since $\OO$ is orbifold simply connected, we have to show, that $G/K_p \hookrightarrow \OO$ is at least $2$ connected. This follows from Lemma \ref{fixpoint:lem:doubledisc}.
\end{proof}
\begin{lem}\label{fixpoint:lem:Lconnected}
    If $L$ is connected and $k \ge 2$, then $\Gamma_q =1$.
\end{lem}
\begin{proof}
     We have $G/L \cong \bb S^k/\Gamma_q$. Hence $G/L$ is a homogeneous space form and the claim follows from the classification.
\end{proof}
\begin{lem}\label{fixpoint:lem:p}
The action $K_p \acton \bb S^l$ is either transitive or the orbits are of dimension $0$, $1$ or $3$.

\end{lem}
\begin{proof}
        Since $\bb S^l/K_p$ is an orbifold, we have by \cite[Corollary~1.2]{WilkingLytchak}, that the action is transitive or the Riemannian foliation induced by the group action has leaves of dimension $0$, $1$, $3$ or $7$. If the leaves are of dimension $7$, then again by \cite[Corollary 1.2]{WilkingLytchak}, the foliation is given by the Hopf fibration $\bb S^{15} \to \bb S^8$, but this cannot be induced by a group action \cite[p. 236]{Gromoll-Grove}.
\end{proof}
\begin{lem}\label{fixpoint:lem:transsphere}
    If $p$ is a fixed point $l \ge 2$ and  $K_p \acton \bb S^l$ transitively, then $\OO \cong \bb S^n$ with a linear cohomogeneity one action.
\end{lem}
\begin{proof}
    From the assumptions we have $\dim\OO = l+1$ and $F$ is also a fixed point. By Lemma \ref{fixpoint:lem:doubledisc} both $\Gamma_p$ and $\Gamma_q$ are trivial, since the suborbifolds $\ast_{\Gamma_p}$ and $\ast_{\Gamma_q}$ must be orbifold simply connected. Therefore $\OO = \bb D^{l+1}\cup \bb D^{l+1}$ is a sphere with a linear action.
    \end{proof}
\begin{rem}
    By Lemma \ref{fixpoint:lem:p}, the assumptions of Lemma \ref{fixpoint:lem:transsphere} are fulfilled, if $p$ is a fixed point and the orbit dimension of the $K_p$-action is larger than $3$.
\end{rem}
\begin{lem}\label{fixpoint:lem:Hpinj}
    The restriction $H_p \to L$ is always injective. 
\end{lem}
\begin{proof}
    Since $H_p$ is the principal isotropy group of $K_p \acton \bb S^l$, it cannot contain any non trivial element of $\Gamma_p$.  Otherwise, $\Gamma_p \acton \bb S^l$ has a kernel, which violates the effectiveness of $\OO$. Therefore $H_p \to L$ is injective. 
\end{proof}
\begin{lem}\label{fixpoint:lem:Kp=Hp}
    If $K_p = H_p$ is connected and $k \ge 2$, then $K = G$ and $\OO$ is diffeomorphic to a sphere $\bb S^{k+l+2}$ with a linear action given by:
    $$g \ast(a_0,\ldots,a_{k+l+1}) = (g\ast(a_0, \ldots,a_{k}),a_{k+1},\ldots,a_{k+l+1})$$
\end{lem}
\begin{proof}
    Since $K_p = H_p$, we have $\Gamma_p = 1$ and therefore $H_p = L$. Since $L = H_p$ is connected, we also have $\Gamma_q =1$. This fully recovers the double disc bundle structure of the given action. Hence they must coincide by the Uniqueness Lemma. 
\end{proof}

By Lemma \ref{fixpoint:lem:connectedgroup} the group $K_p$ is connected, in all cases, but $K = \S^1$. We will show later, that this is also true in the latter case, given, that $G$ acts with trivial kernel.\\
Case $0$: $K = \S^1$. In this case $k = 1$ and $\codim F=2$. Let $\Gamma_p$ be the local group at $G\cdot p$. Then there is an extension $1 \to \Gamma_p \to K_p \to G_p\to 1$, such that $D(G\cdot p) \cong \bb D^{l+1} \times_{K_p} G$. First, let $G_p = \S^1$ and hence the dimension of $\OO$ is even. We first notice, that $K_q$ must be connected: $K_q$ is contained in $\O(2)$ with $K = \SO(2)$. Since $\Gamma_q$ is the local group of a intrinsically regular point, we have $\Gamma_q\subseteq\O(2)$. Because $\Gamma_q$ cannot contain any reflection, we have $\Gamma_q \subseteq \SO(2) = K$. Hence $K_q = K$ is connected and $\Gamma_q$ is cyclic. Now let $\dim \OO = 2$. Then by the same argument as before $K_p$ is connected and $\Gamma_p$ is cyclic. This gives the structure of a weighted $\bb C\bb P^1$. Hence let $\dim \OO \ge 4$. In this case $l \ge 3$ and hence $F$ is orbifold simply connected by Lemma \ref{fixpoint:lem:doubledisc}. We claim that $K_p$ is connected as well. 
 By Lemma \ref{fixpoint:lem:Hpinj}, $H_p \hookrightarrow L = \lbrace e\rbrace$. Hence $H_p$ is trivial and therefore $K_p \acton\bb S^{2l-1}$ effectively and we have $\bb S^{2l-1}/K_p \cong F_{\eff}$. Since $\pi_1^{\orb}(F_\eff) \cong \lbrace 1\rbrace$, also $K_p$, must be connected, by the long exact sequence of homotopy groups. Hence $K_p \cong \S^1$ and $\Gamma_p$ is cyclic. $\OO$ must be diffeomorphic to a weighted complex projective space. \\
From now on $K_p$ and hence $G_p$ must be connected by lemma \ref{fixpoint:lem:connectedgroup}. If $p$ is a fixed point and $\dim K_p/H_p \neq 3$ , then $\OO$ is equivariantly diffeomorphic to a sphere with a linear cohomogeneity one action by lemma \ref{fixpoint:lem:transsphere}. If $K_p = H_p$, then $\OO$ is equivariantly diffeomorphic to a sphere with a linear suspension action by lemma \ref{fixpoint:lem:Kp=Hp}. Hence we assume, that $p$ is no fixed point except if $ \dim K_p/H_p = \dim G/L = k = 3$ and that $K_p \neq H_p$.\\

\noindent
Case $2$: $K= \SU(k+1)$, and $H = \SU(k)$ for $k \ge 2$. Since $\codim F = 2k+2$, we need to compute the bundle $D(G\cdot p)$. Let $H_p$ be the principal isotropy of the $K_p$-action on $\bb S^l$. Since $\dim \SU(k+1)/\SU(k) = 2k+1 \ge 5$, we have $K_p/H_p \cong \S^1$. $G_p \cong \U(k) = \S^1 \times_{\bb Z_k}\SU(k)$ and $K_p$ is a connected covering of $G_p$ and hence $\Gamma_p =\bb Z_{p^\prime}$ is cyclic. Since $K_p \to G_p$ is a $p^\prime$-fold covering, we have $K_p = \S^1 \times_{\bb Z_k}\SU(k)$, but this time $\bb Z_k \to \bb Z_{\frac{k}{\gcd(k,p^\prime})}\subseteq\SU(k)$. Let $K_p/H_p$ acts with weights $q_0^\prime,\ldots,q_n^\prime$ and $\bb Z_g$ be the kernel of the action of the $\S^1$-factor in $K_p$. Then $\OO$ is equivariantly diffeomorphic to a weighted complex projective space $\C\bb P^{k+1+n+1}[p^\prime, \ldots, p^\prime, q_0, \ldots, q_n]$, with $q_i = q_i^\prime \cdot g$.  \\

\noindent
Case 3: $(K,H) = (\Sp(k+1),\Sp(k))$ and hence $\codim(F) = 4k+4$. Therefore $G = \Sp(k+1),\Sp(k+1)/\lbrace \pm \Id\rbrace$. Hence we can assume $G = \Sp(k+1)$ acts possibly with kernel $\bb Z_2$. Since $K_p$ is connected  and $\dim \Sp(k+1) /\Sp(k) = 4k+3$ there are two cases:
\begin{enumerate}
    \item $\dim K_p/H_p =3$.
    \item $K_p/H_p = \S^1$.
\end{enumerate}
Subcase 1.: $K_p/(H_p)_0 \cong \Sp(1),\SO(3)$. Since any linear action of $\SO(3)$ on a sphere has nontrivial stabilizers, we have $K_p/H_p \cong \Sp(1)$ and the action on $\bb R^{l+1}$ is of quaternionic type. Hence $l+1=4m$ and $\bb R^{4m} = \bigoplus_{i = 0}^r k_iV_{2n_i+1}$ and $H_p$ is connected. Furthermore $\Gamma_p =1$, since $K_p = \Sp(1) \times \Sp(k)$ is simply connected and $G/G_p = \bb H\bb P^k$. 
Hence $L$ is connected and thus $\Gamma_q = 1$. We have recovered the action of $\Sp(k+1)$ on a weighted quaternionic projective space. \\
Subcase 2.: We can assume $G = \Sp(k+1)$ possibly acting with kernel $\bb Z_2 = \lbrace \pm \Id\rbrace$. In any case $G_p = \S^1\times \Sp(k)$ and therefore $\Gamma_p = \bb Z_{p^\prime}$ for some $p^\prime \in \bb Z$ and $K_p = \S^1 \times \Sp(k)$. The $\Sp(k)$-factor acts trivially on $\bb S^l$, hence $H_p = \bb Z_g \times \Sp(k)$. Since $H_p \hookrightarrow L$ is injective, $g$ and $p ^\prime$ are  relatively prime. Therefore there exist weights $q_0, \ldots,q_l$, with $\gcd(q_0, \ldots, q_l) =1$, of the action of the $\S^1$-factor. Hence we recover the $\Sp(k+1)$-action on $\bb C\bb P^{k+1 +l}[p^\prime,\ldots,p^\prime,q_0,\ldots,q_l]$.\\

\noindent
Case 4.: $(K,H) = (\Spin(9), \Spin(7))$ and $k = 15$. $H$ is embedded in $\Spin(9)$ via the spin-representation. Either $G = \Spin(9)$ or $G = \SO(9)$. Hence we can assume that $G = \Spin(9)$, possibly acting with ineffective kernel $\bb Z_2$. Since $K_p$ is connected and $k = 15$, we have $K_p = \Spin(8),\SO(8)$. We have $\Gamma_p = 1$. By Lemma \ref{fixpoint:lem:p}, the action of $K_p$ on $\bb S^l$ is transitive. Hence $H_p = L = \Spin(7)$ and $F$ is a point. Since $L$ is connected, we also have $\Gamma_q = 1$ and thus $\OO$ is a manifold. As in \cite{Grove1997-xz}, the normal sphere bundle $E \to \Spin(9)\cdot p$ is $\Spin(9)$-equivariantly equivalent to the Hopf-fibration:
$$\S^7 = \Spin(8)/\Spin(7) \to \bb S^{15} = \Spin(9)/\Spin(7) \to \bb S^8 = \Spin(9)/\Spin(8) $$
And thus $\OO$ is $\Spin(9)$-equivariantly diffeomorphic to the Cayley plane ${\rm Ca}\bb P^2$.\\

\noindent
Case 5.: $(K,H) = (\Spin(7),G_2)$ and we can assume $G = \Spin(7)$. Since $G_2$ is a maximal connected group in $\Spin(7)$ and $k = 7$ there are no cases left to prove. 

\noindent
Case 6.: $(K,H) = (G_2,\SU(3))$ and we assume $G = G_2$, which has trivial center. Again $\SU(3)$ is a maximal connected subgroup of $G_2$ and $k = 6$. Hence there are no cases left to prove.

\noindent Case 2.1 now follows from case 0 if $k =0$ and from case 2 if $k \ge 1$. Cases 3.1 and 3.2 follow from case 3. 

\section{Ineffective Orbifolds}\label{sec:ineffective}

Let $\mathcal{N} =(G_1 \garrow G_0)$ be an ineffective orbifold with principal local group $K$. We get an effective orbifold in the following way: We define the normal subgroupoid $\mathcal{K} = (K_1 \garrow K_0)$: Let $K_0 = G_0$. Furthermore, let $x \in G_0$. We may assume $\OO$ to be an \'etale groupoid, so any arrow $g \in G_x$ acts as a local diffeomorphism on an open neighborhood of $x$. Hence, we get an isomorphism $dg \colon T_xG_0 \to T_x G_0$. This induces a homomorphism $d \colon G_x \to GL(T_xG_0)$. We define $K_x \coloneqq \ker(d) $. Furthermore we set $K_1 \coloneqq \bigsqcup_{x \in G_0} K_x$, with the topology induced by inclusion. With this data $K_1 \garrow K_0$ is a proper \'etale Lie groupoid, with source and target maps being restriction from $\OO$. $\mathcal{K}$ is a normal subgroupoid, since any arrow $h \colon x \to y$ in $G_1$ induces a diffeomorphism of open neighborhoods of $x$ and $y$ and hence a map $dh\colon T_xG_0 \to T_y G_0$. If $g \in K_x$, then $dg = \Id_x$ and hence $d(hgh^{-1}) = dh \circ dg \circ dh^{-1} = dh\circ dh^{-1} = \Id_y$. Therefore $hgh^{-1} \in K_y$ and $K_x \cong K_y$. Furthermore all $K_x$ in the same connected component of $K_0 = G_0$ are isomorphic to each other, say to a fixed group $K$. The quotient $\OO = \mathcal{N}/\mathcal{K}$ is an effective orbifold. 

Now we turn to the question of how many choices of $\mathcal{N}$ are possible for fixed $\OO$ and $K$, i.e. to the problem of classifying ineffective extensions.

Given $g:x\to y \in G_1$, $k\mapsto g\circ k \circ g^{-1}$ is an isomorphism $K_x\to K_y$. This induces a morphism $\mathcal{N}\to Aut(\mathcal{K})$, further inducing a morphism $\mathcal{O}\to Out(\mathcal{K})$ after quotienting by inner automorphisms of $K$. 

Assuming $\mathcal{O}$ is orbifold-simply connected, we have a look at the induced map of classifying spaces  $B\mathcal{O}\to B\operatorname{Out}\mathcal{K}$. Since $\operatorname{Out}\mathcal{K}$ is connected as a groupoid (i.e. there is an arrow between any two objects, since $K_x$ is isomorphic to a fixed group $K$ at every point $x$), the inclusion of any point induces a groupoid equivalence between $\operatorname{Out}\mathcal{K}$ and the one point groupoid $\operatorname{Out}K$. Hence the classifying spaces are (weakly) homotopy equivalent and we have $[B\mathcal{O},B\operatorname{Out}\mathcal{K}]=[B\mathcal{O},B\operatorname{Out}K]$. in particular, since $B\operatorname{Out}K$ is an Eilenberg-MacLane space $K(\operatorname{Out}K, 1)$, we have $[B\mathcal{O},B\operatorname{Out}\mathcal{K}]=\Hom(\pi_1B\mathcal{O}, \operatorname{Out}K)=0$ since $\pi_1B\mathcal{O}=0$. 

In particular, every map $B\mathcal{O}\to B\operatorname{Out}\mathcal{K}$ is null-homotopic, and we get that the band of every extension is trivial. (For definition of band and triviality thereof, see \cite{Laurent-Gengoux_Stienon_Xu_2009}[Def 3.4/3.8]). Thus, by  \cite{Laurent-Gengoux_Stienon_Xu_2009}[Thm 3.13], extensions are classified by classes in $H^2_{orb}(\OO, Z(K))$. by universal coefficients and Hurewicz, $H^2_{orb}(\OO, Z(K))\cong \Hom(H_2^{orb}(\mathcal{O}), Z(K))\cong \Hom(\pi_2^{orb}(\mathcal{O}), Z(K))$, using again $\pi_1(B\OO)=1$.

By Quillen's Theorem B\cite{QuillenHAKT}, our extension of groupoids induces a long exact sequence of homotopy groups
\begin{align} \label{quillensequence}
    \ldots \to \pi_2^\orb(\OO) \xrightarrow{\partial} K \to \pi_1^\orb(\NN) \to \pi_1^\orb(\OO) = 1.
\end{align}

If we look only for those extensions where $\pi_1^\orb(\NN) = 1$, then $\partial \colon \pi_2^\orb(\OO) \to K$ must be surjective and hence $K = Z(K)$ abelian. 
An immediate consequence is:
\begin{lem}
    Let $\NN$ be an orbifold-simply connected ineffective orbifold. Then the generic local group is abelian.
\end{lem}
We want to compute the (simply connected) ineffective structures over weighted projective spaces. First we compute their second orbifold homotopy groups:
\begin{lem}\label{orbifolds:orbifoldhomotopyCROSS}
    We have \begin{enumerate}
        \item  $\pi_2^\orb(\bb S^n) = \pi_2(\bb S^n) = 0$ for $n \ne 2$
        \item $\pi_2^\orb(\bb C \bb P(W)) = \bb Z$
        \item $\pi_2^\orb(\bb H \bb P(V)) = 0$
        \item $\pi_2^\orb(\Ca \bb P^2) = \pi_2(\Ca\bb P^2) = 0$
    \end{enumerate}
\end{lem}
\begin{proof}
    1. and 4. are clear. For the other two, let $\S^l$ with $l = 1,3$ act linearly and almost freely on a sphere $\bb S^N$, with $N(l) = 2n+1$ if $l = 1$ and $N = 4m +3$ if $l = 3$. In both cases $N \ge 3$. We consider the fibration
    $$\bb S^N \to B(\bb S^N/S^l) = ES^l \times_{\S^l} \bb S^N \to BS^l.$$
    
    Which gives a long exacts sequence of homotopy groups:
    $$\ldots\to \pi_2(\bb S^N) \to \pi_2^\orb(\bb S^n/\S^l) \to \pi_2(B\S^l) \to \pi_1(\bb S^N)\to \ldots$$
    Since $N \ge 3$, we have $\pi_1(\bb S^N) \cong \pi_2(\bb S^N) \cong 0$. Hence $$\pi_2^\orb(\bb S^N/\S^l) \cong \pi_2(B\bb\S^l).$$
    For $l = 1$, we have $\bb S^N/\S^1 \cong \bb C \bb P(W)$ and $B\S^1 = \bb C \bb P ^\infty$. Hence $\pi_2^\orb(\bb C \bb P(W)) \cong \pi_2( BS^1) = \bb Z$.\\
    For $l = 3$, we have $\bb S^N/\S^3 \cong \bb H \bb P(V)$ and $B \S^3 \cong \bb H \bb P ^\infty$. Hence we have $\pi_2^\orb(\bb H \bb P(V)) \cong 0$.
\end{proof}
We have the following Corollary:
\begin{cor}\label{orbifolds:ineffectiveCROSS} \mbox{}
    \begin{enumerate}
        \item The only ineffective orbifolds $\OO$ with effective model $\OO_\eff$ isomorphic to $\bb S^n$, $n \ge 3$, $\bb H \bb P (V)$ and $\Ca \bb P^2$ are trivial, i.e. given by a trivial action of a finite group $K$ on $\OO_\eff$. Consequently there are no simply connected ineffective such orbifolds.
        \item Let $\OO$ be a simply connected orbifold with $\OO_\eff \cong \bb C\bb P(W)$, then $\OO$ has generic local group $\bb Z_k$ and is given as the quotient of the following action: $$\S^1 \acton \bb S(W), z \ast a = z^k\cdot a,$$
        where $\cdot$ denotes the original action defining $\bb C \bb P(W)$.
        \item Let $\OO$ be a not necessarily simply connected orbifold with $\OO_\eff \cong \bb C\bb P(W)$, and let $K$ be the generic local group. Let $\psi \colon \bb Z \to \S^1$ be $1\mapsto e^{2\pi i /k}$. Then $\OO$ is given by  $$K {}_\phi \times_{\psi} \S^1 \acton \bb S(W), [g,z] \ast a = z^k\cdot a$$ for some homomorphism $\phi \colon \bb Z \to Z(K) \subset K$,
        where $\cdot$ denotes the original action defining $\bb C \bb P(W)$. 
    \end{enumerate}
\end{cor}
\begin{proof}Let $K$ be the generic local group of $\OO$. 
    \begin{enumerate}
        \item By Lemma \ref{orbifolds:orbifoldhomotopyCROSS}, we have that $H^2_\orb (\OO,Z(K)) \cong \Hom(\pi_2^\orb(\OO),Z(K)) = 0$. Hence there is only one ineffective structure, which must be trivial. 
        \item  Since $\OO$ is simply connected, we have that $\partial \colon \pi_2^\orb(\bb C \bb P(W)) = \bb Z \to K$ must be surjective. Therefore $K$ must be cyclic and hence $K \cong \bb Z_k$. Furthermore, we have $H^2_\orb(\OO,Z(K)) \cong \Hom(\bb Z,\bb Z_k) \cong \bb Z_k$. The surjective homomorphisms are given by $\bb  Z_k^\times = \lbrace n \in \bb Z_k\setvert \gcd(k,n)=1\rbrace $, which the outer automorphism group of $\bb Z_k$ acts transitively on. Hence there is up to an isomorphism of the outer automorphism group only a single nontrivial ineffective structure with $\OO_\eff  = \bb C\bb P(W)$. Since the previously defined action induces a simply connected ineffective structure, with generic isotropy $\bb Z_k$, the claim follows.
        \item By the long exact sequence of homotopy groups for the fibration defining $\cp(W)$, a generator of $\pi_2^{orb}(\cp(W))$ is induced by $\pi_1S^1$. Thus we see that our construction for a given $\phi$ produces a long exact sequence \eqref{quillensequence} with boundary map $\partial$ equal to $\phi$. Thus our construction covers all classes in $H^2_{orb}(\mathcal{O}, Z(K))$.
    \end{enumerate}
\end{proof}

\printbibliography

\end{document}